\documentclass[12pt,oneside,american,english]{amsart}
\usepackage[T1]{fontenc}
\usepackage[latin9]{inputenc}
\usepackage{geometry}
\usepackage[active]{srcltx}
\usepackage{babel}
\usepackage{refstyle}
\usepackage{units}
\usepackage{amsthm}
\usepackage{amstext}
\usepackage{amssymb}
\usepackage[all]{xy}
\PassOptionsToPackage{normalem}{ulem}
\usepackage{ulem}
\usepackage[unicode=true,pdfusetitle,
 bookmarks=true,bookmarksnumbered=false,bookmarksopen=false,
 breaklinks=false,pdfborder={0 0 0},backref=false,colorlinks=false]
 {hyperref}
\usepackage{breakurl}

\makeatletter

\AtBeginDocument{\providecommand\secref[1]{\ref{sec:#1}}}
\AtBeginDocument{\providecommand\lemref[1]{\ref{lem:#1}}}
\AtBeginDocument{\providecommand\thmref[1]{\ref{thm:#1}}}
\AtBeginDocument{\providecommand\exaref[1]{\ref{exa:#1}}}
\AtBeginDocument{\providecommand\propref[1]{\ref{prop:#1}}}
\AtBeginDocument{\providecommand\subref[1]{\ref{sub:#1}}}
\AtBeginDocument{\providecommand\remref[1]{\ref{rem:#1}}}
\AtBeginDocument{\providecommand\corref[1]{\ref{cor:#1}}}
\AtBeginDocument{\providecommand\defref[1]{\ref{def:#1}}}
\AtBeginDocument{\providecommand\eqref[1]{\ref{eq:#1}}}
\providecommand{\tabularnewline}{\\}
\RS@ifundefined{subref}
  {\def\RSsubtxt{section~}\newref{sub}{name = \RSsubtxt}}
  {}
\RS@ifundefined{thmref}
  {\def\RSthmtxt{theorem~}\newref{thm}{name = \RSthmtxt}}
  {}
\RS@ifundefined{lemref}
  {\def\RSlemtxt{lemma~}\newref{lem}{name = \RSlemtxt}}
  {}

\numberwithin{equation}{section}
\numberwithin{figure}{section}
\numberwithin{table}{section}
\theoremstyle{plain}
\newtheorem{thm}{\protect\theoremname}[section]
  \theoremstyle{definition}
  \newtheorem{defn}[thm]{\protect\definitionname}
  \theoremstyle{plain}
  \newtheorem*{thm*}{\protect\theoremname}
  \theoremstyle{plain}
  \newtheorem{lem}[thm]{\protect\lemmaname}
  \theoremstyle{remark}
  \newtheorem{rem}[thm]{\protect\remarkname}
  \theoremstyle{plain}
  \newtheorem{prop}[thm]{\protect\propositionname}
  \theoremstyle{definition}
  \newtheorem{example}[thm]{\protect\examplename}
  \theoremstyle{plain}
  \newtheorem*{lem*}{\protect\lemmaname}
  \theoremstyle{plain}
  \newtheorem{cor}[thm]{\protect\corollaryname}
  \theoremstyle{definition}
  \newtheorem*{example*}{\protect\examplename}

\newref{ex}{name=Example~}
\newref{sub}{name=Subsection~}
\newref{cor}{name=Corollary~}
\newref{rem}{name=Remark~}
\newref{prop}{name=Proposition~}
\newref{thm}{name=Theorem~}
\newref{def}{name=Definition~}

\makeatother

  \addto\captionsamerican{\renewcommand{\corollaryname}{Corollary}}
  \addto\captionsamerican{\renewcommand{\definitionname}{Definition}}
  \addto\captionsamerican{\renewcommand{\examplename}{Example}}
  \addto\captionsamerican{\renewcommand{\lemmaname}{Lemma}}
  \addto\captionsamerican{\renewcommand{\propositionname}{Proposition}}
  \addto\captionsamerican{\renewcommand{\remarkname}{Remark}}
  \addto\captionsamerican{\renewcommand{\theoremname}{Theorem}}
  \addto\captionsenglish{\renewcommand{\corollaryname}{Corollary}}
  \addto\captionsenglish{\renewcommand{\definitionname}{Definition}}
  \addto\captionsenglish{\renewcommand{\examplename}{Example}}
  \addto\captionsenglish{\renewcommand{\lemmaname}{Lemma}}
  \addto\captionsenglish{\renewcommand{\propositionname}{Proposition}}
  \addto\captionsenglish{\renewcommand{\remarkname}{Remark}}
  \addto\captionsenglish{\renewcommand{\theoremname}{Theorem}}
  \providecommand{\corollaryname}{Corollary}
  \providecommand{\definitionname}{Definition}
  \providecommand{\examplename}{Example}
  \providecommand{\lemmaname}{Lemma}
  \providecommand{\propositionname}{Proposition}
  \providecommand{\remarkname}{Remark}
  \providecommand{\theoremname}{Theorem}
\providecommand{\theoremname}{Theorem}

\begin{document}
\global\long\def\sp{\mbox{Span}_{F}}
\global\long\def\Aut{\mbox{Aut}}
\global\long\def\Hom{\mbox{Hom}_{F}}
\global\long\def\Tr{\mbox{Tr}}
$ $\global\long\def\End{\mbox{End}}
\global\long\def\htt{\mbox{ht}}
\global\long\def\Cent{\mbox{Cent}}

\title{On the Codimension Sequence of $G$-Simple Algebras}

\author{Yakov Karasik }

\author{Yuval Shpigelman}

\address{Department of Mathematics, Technion - Israel Institute of Technology,
Haifa 32000, Israel.}

\email{yuvalshp 'at' tx.technion.ac.il (Y. Shpigelman),}

\email{yaakov 'at' tx.technion.ac.il (Y. Karasik).}

\begin{abstract}
Let $F$ be an algebraically closed field of characteristic zero and
let $G$ be a finite group.\foreignlanguage{american}{ In this paper
we will show that the asymptotics of $c_{n}^{G}(A)$, the $G$-graded
codimension sequence of a finite dimensional $G$-simple $F$-algebra
$A$, has the form $\alpha n^{\frac{1-\dim_{F}A_{e}}{2}}(\dim_{F}A)^{n}$
(as conjectured by E.Aljadeff, D.Haile and M. Natapov), where $\alpha$
is some positive real number and $A_{e}$ denotes the identity component
of $A$. In the special case where $A$ is the algebra of $m\times m$
matrices with an arbitrary elementary $G$-grading we succeeded in
calculating the constant $\alpha$ explicitly.}
\end{abstract}

\maketitle
\selectlanguage{american}%
\tableofcontents{}

\selectlanguage{english}%

\section*{Introduction}

\selectlanguage{american}%
Let $A$ be an associative algebra PI algebra over a field $F$ of
characteristc zero. The $T$-ideal of polynomial identities $Id(A)$
(inside the free algebra $F\left\langle X\right\rangle $ over $F$
on a countable set of variables) is an important invariant of $A$
which attracted a lot of attention in the last 60 years or so. For
instance, it plays an important role in the study of growth properties
of $A$ \foreignlanguage{english}{(see e.g. \cite{key-15,key-19,key-20,key-10})}
or in its representation theory (for instance it is known that any
relatively free algebra is representable, see \foreignlanguage{english}{\cite{key-13,key-14}}).
It turns out however, that the precise knowledge of $Id(A)$ where
$A$ is finite dimensional is an extremely hard task. Of course, knowing
a set of generators of $Id(A)$ (it is known by a famous result of
Kemer that any $T$-ideal is finitely generated) would be a major
step ahead but even then it is not clear how to determine explicitly
whether a polynomial is or is not generated by the given set). With
this point of view it is natural (and many times more effective) to
study general invariants attached to $T$-ideals.

One of the most fruitful invariants of this sort (introduced by Regev
in \cite{key-46}) is the codimension sequence attached to a $T$-ideal
of identities and its asymptotic behavior. Let us recall its definition.
Let $Id(A)\subseteq F\left\langle X\right\rangle $ be a $T$-ideal
as above ($X={\{x_{1},x_{2},....\}}$) and let $P_{n}$ denote the
$n!$ dimensional $F$-space spanned by all permutations of the (multilinear)
monomial $x_{1}x_{2}\cdot\cdot\cdot x_{n}$. Define the $n$-th coefficient
$c_{n}(A)$ of the codimension sequence of the $F$-algebra $A$,
by $c_{n}(A)=\dim_{F}(P_{n}/(P_{n}\cap Id(A))$. 

It turns out that computing explicitly the codimension sequence of
an algebra $A$ is also very difficult. Indeed, there are only few
examples of algebras whose codimension sequence is known explicitly.
It is more feasible to study the asymptotics of the codimension sequence
and its exponential component (see definition below). We have collected
some of the main results regarding these invariants. 

\selectlanguage{english}%
Let $A$ be a PI $F$-algebra. It is known (Regev, \cite{key-46})
that the sequence of codimensions is exponentially bounded. Moreover,
a key result of Giambruno and Zaicev says that $\lim\,\sqrt[n]{c_{n}(A)}$
exists, and is a non-negative integer.\foreignlanguage{american}{
We denote the limit by $\exp(A)$ and refer to it as the \emph{exponent}
of $A$. In fact, Giambruno and Zaicev showed that in the case that
the algebra $A$ is finite dimensional over an algebraically closed
field $F$ of zero characteristic, $\exp(A)$ may be interpreted as
the dimension of a suitable semisimple subalgebra of $A$. Following
a different track Berele and Regev (see \cite{key-18E,key-19E}) found
new invariants related to the codimension sequence by proving that
the asymptotics of the codimension sequence of any algebra with $1$
is of the form $\alpha n^{b}d^{n}$, where $\alpha$ is a positive
real number, $b$ is a half integer and $d=\exp(A)$. We refer to
the constants $\alpha$, $b$ and $d$ as the constant part, polynomial
part and exponential factor of the asymptotics respectively. As mentioned
above, the exponential part $d=\exp(A)$ has }a satisfactory interpretation.
On the other hand the polynomial part (and even more the constant
part) are more subtle invariants and in particular no such interpretation
is known.

\selectlanguage{american}%
An important example which was intensively studied by means of PI
theory is the algebra of $m\times m$ matrices over a field $F$ of
characteristic zero. As for general finite dimensional algebras very
little is known about (finite) generating sets of $Id(A)$ (except
the case $m=2$). However the asymptotics of the codimension sequence
was remarkably computed by Regev \cite{key-Regev} using results of
Formanek\cite{key-Formanek}, Procesi\cite{key-Procesi} and Razmyslov\cite{key-Raz}
about invariant theory and the theory of Hilbert series. \foreignlanguage{english}{In
this paper, }we extract from their work a program for calculating
the codimension asymptotic of finite dimensional algebras. The basic
idea is explained in \secref{Motivation}. This program might be considered
as a step towards finding an interpretation of the polynomial and
constant part of the codimension sequence of finite dimensional algebras.
Nevertheless, it is still seems very challenging to tackle (finite
dimensional) algebras with a non zero radical. The compromise is to
change the framework obtaining a rich family of examples which are
yet manageable. More specifically, we work with algebras graded by
a finite group $G$ and consider their $G$-graded polynomial identities
and $G$-graded codimension sequence. Let us recall the necessary
definitions.

\selectlanguage{english}%
Suppose now $A$ is a $G$-graded algebra, where $G$ is a finite
group. Let $X^{G}$ be a countable set of variables $\left\{ x_{i,g}:g\in G;i\in N\right\} $
and let $F\left\langle X^{G}\right\rangle $ be the free algebra on
the set $X^{G}$. Clearly, the algebra $F\left\langle X^{G}\right\rangle $
is $G$-graded in a natural way (the degree of a monomial is the element
of $G$ which is equal to the product of the homogeneous degrees of
its variables). Given a polynomial in $F\left\langle X^{G}\right\rangle $
we say that it is a $G$-graded identity of $A$ if it vanishes upon
any admissible evaluation of $A$, that is, a variable $x_{i,g}$
assumes values only from the corresponding homogeneous component $A_{g}$.
The set of all $G$-graded identities is an ideal in the free $G$-graded
algebra which we denote by $Id^{G}(A)$. Moreover, $Id^{G}(A)$ is
a $T$-ideal, namely, it is closed under $G$-graded endomorphisms
of $F\left\langle X^{G}\right\rangle $. In an analogy to the ungraded
case, we let\foreignlanguage{american}{ $P_{n}^{G}$ denote the $|G|^{n}\cdot n!$
dimensional $F$-space spanned by all permutations of the (multilinear)
monomial $x_{1,g_{1}},x_{2,g_{2}}\cdot\cdot\cdot x_{n,g_{n}}$, where
$g_{i}\in G$, and define the $n$-th coefficient $c_{n}^{G}(A)$
of the codimension sequence of a $G-$graded algebra $A$ by $c_{n}^{G}(A)=\dim_{F}(P_{n}^{G}/(P_{n}^{G}\cap Id^{G}(A))$.}
As in the ungraded case Aljadeff, Giambruno and La-Matina, in different
collaborations (see \cite{key-A+G+L,key-G+L,key-A+G}) showed that
$\lim\sqrt[n]{c_{n}^{G}(A)}$ exists and is a non-negative integer
denoted by $\exp^{G}(A)$. Moreover, they showed that if $A$ is finite
dimensional, then there exists a $G$-graded semisimple subalgebra
whose dimension interprets $\exp^{G}(A)$. It was also conjectured
that the asymptotics of the codimension sequence has the same structure
as in the ungraded case (constant $\times$ polynomial part $\times$
exponential part), and indeed the second author of this article proved
this in \cite{Shpigelman} for affine $G$-graded algebras with 1.

When considering finite dimensional $G$-simple algebras (over an
algebraically closed field $F$ with characteristic $0$) two different
examples arise. 
\begin{enumerate}
\item \textbf{Fine grading}: Let $A$ be a twisted group algebra $A=F^{\mu}G$,
where $\mu$ is a 2-cocycle of $G$ and the grading is given by $A_{g}=Fb_{g}$.
This is indeed a $G$-simple algebra since every homogenous component
contains an invertible element. Notice that each component is one
dimensional. 
\item \textbf{Elementary grading}:\textbf{ }Now assume $A$ is the matrix
algebra $A=M_{m}(F)$. Consider a vector $\mathfrak{g}=(\gamma_{1},...,\gamma_{m})\in G^{\times m}$
and define a grading on $A$ by declaring $A_{g}=\mbox{Span}_{F}\{e_{i,j}|g=\gamma_{i}^{-1}\gamma_{j}\}$,
where $e_{i.j}$ are the elementary matrices. The resulting $G$-graded
algebra is (ungraded) simple, so surely $G$-graded simple. 
\end{enumerate}
In \cite{key-5} it is shown that, in fact, every $G$-simple finite
dimensional $F$-algebra $A$ is a ``combination'' of the two examples
above (see also \secref{Review-of-simple} for the precise statement). 

A major part of this article (sections \ref{sec:Invariants},\ref{sec:Hilbert-series}
and \ref{sec:SI=00003DC}) is devoted to finding the asymptotics of
the codimension sequence of the matrix algebra $A=M_{m}(F)$ graded
elementary.
\begin{defn}
Suppose $a_{n}$ and $b_{n}$ are two series of real positive numbers.
We write $a_{n}\sim b_{n}$ if 
\[
\lim_{n\to\infty}\frac{a_{n}}{b_{n}}=1.
\]

\end{defn}
The main result is:
\begin{thm*}[\textbf{A}]
 Let $A$ be the $F$-algebra of $m\times m$ matrices with elementary
$G$-grading, and let $\mathfrak{g}=(\gamma_{1},...,\gamma_{m})$
be the grading vector where all the $\gamma_{i}$'s are distinct.
Then 
\[
c_{n}^{G}(A)\sim\alpha n^{\frac{1-\dim_{F}A_{e}}{2}}(\dim_{F}A)^{n}
\]
where
\[
\alpha=\frac{1}{|H_{\mathfrak{g}}|}m^{\frac{1-\dim_{F}A_{e}}{2}+2}\left(\frac{1}{\sqrt{2\pi}}\right)^{m-1}\left(\frac{1}{2}\right)^{\frac{1-\dim_{F}A_{e}}{2}}\prod_{i=1}^{k}\left(1!2!\cdots(m_{i}-1)!m_{i}^{-\frac{1}{2}}\right).
\]
The subgroup $H_{\mathfrak{g}}\leq G$ is defined in \secref{Invariants}
and $m_{1},...,m_{k}$ are the non-zero multiplicities of distinct
elements of the set $G/H$ inside the vector $(\gamma_{1}H,....,\gamma_{m}H)$.
\end{thm*}
The asymptotic behavior of the codimension sequence in the fine grading
case was already established in \cite{key-A+K}. However using our
own methods we recalculate it in \secref{Fine}. Finally, in \secref{General-case}
we combine the results in the first sections to capture the polynomial
part (the exponential part is of course known) of the codimention
asymptotics of a general finite dimensional $G$-graded algebra:
\begin{thm*}[\textbf{B}]
 For every finite dimensional $G$-simple algebra $A$, there is
a constant $\alpha$ such that
\end{thm*}
\[
c_{n}^{G}(A)\sim\alpha n^{\frac{1-dimA_{e}}{2}}\left(\dim_{F}A\right)^{n}.
\]

In particular every $G$-simple, finite dimensional algebra $A$ has
a polynomial part equal to $\frac{1-\dim_{F}A_{e}}{2}$. This was
conjectured by E.Aljadeff, D.Haile and M. Natapov (see \cite{key-11}).

\section{\label{sec:Motivation}Getting started}

In order to find $c_{n}^{G}(A)$ asymptotically we will consider another
series of spaces whose corresponding series of dimensions is equal
asymptotically to $c_{n}^{G}(A)$ with the advantage of being asymptotically
computable. The goal of this section is to introduce and motivate
those approximating spaces. 

Let $A$ be any finite dimensional $F$-algebra. It is easy to obtain
a crude upper bound of $c_{n}^{G}(A)$: This is done by observing
that:
\begin{lem}
$C_{n}^{G}(A)=P_{n}^{G}/Id^{G}(A)\cap P_{n}^{G}$ can be embedded
into $\Hom(A^{\otimes n},A)$ (linear homomorphisms of linear $F$-spaces,
where $A^{\otimes n}=A\otimes\cdots\otimes A$ - $n$ times) via 
\[
\psi\left(x_{1,g_{1}}\cdots x_{n,g_{n}}\right)(a_{1}\otimes\cdots\otimes a_{n})=(a_{1})_{g_{1}}\cdots(a_{n})_{g_{n}}
\]
where $(a)_{g}$ is the projection of $a$ onto $A_{g}$ along $\sum_{h\neq g}A_{h}$.
Thus, $c_{n}^{G}(A)\leq(\dim_{F}A)^{n+1}$ . \end{lem}
\begin{proof}
Notice that $\psi$ is well defined since for any $G$-graded multilinear
identity $f=f(x_{1,g_{1}},...,x_{n,g_{n}})$ of $A$, 
\[
\psi\left(f\right)(a_{1}\otimes\cdots\otimes a_{n})=f\left((a_{1})_{g_{1}},...,(a_{n})_{g_{n}}\right)=0,
\]
since $f\in Id^{G}(A)$. 

Next, if for some $G$-graded multilinear polynomial $f=f(x_{1,g_{1}},...,x_{n,g_{n}})$
the image of $f$ under $\psi$ is zero, we get that for every $a_{1}\in A_{g_{1}},...,a_{n}\in A_{g_{n}}$:
\[
0=\psi\left(f\right)(a_{1}\otimes\cdots\otimes a_{n})=f\left((a_{1})_{g_{1}},...,(a_{n})_{g_{n}}\right)=f(a_{1},...,a_{n}).
\]
That is $f\in Id^{G}(A)$. Thus, $\psi$ is an embedding. 

Finally, $c_{n}^{G}(A)\leq(\dim_{F}A)^{n+1}$ follows, because $\dim_{F}\left(\Hom(A^{\otimes n},A)\right)=\left(\dim_{F}A\right)^{n+1}$.
\end{proof}
In most cases this bound is far from being strict, since for most
$G$-graded algebras $A$ one has $\exp^{G}(A)<\dim_{F}A$. It is
better in the finite dimensional $G$-simple case, since then the
exponent is equal to $\dim_{F}A$. Still, it makes more sense to consider
subspaces which, in some sense, take into account the ($G$-graded)
structure of $A$. 

More precisely, denote $\widetilde{G}=\Aut_{G}(A)$, that is $\phi\in\widetilde{G}$
if $\phi\in\Aut(A)$ and $\phi(A_{g})\subseteq A_{g}$ for every $g\in G$
(in other words, the ($F$-algebra) automorphisms of $A$ which preserve
the $G$-grading). The action of $\widetilde{G}$ on $A$ and thus
on $A^{\otimes n}$ (diagonally) induces an action on the spaces $\Hom(A^{\otimes n},A)$
via 
\[
\phi(f)(a_{1}\otimes\cdots\otimes a_{n})=\phi\circ f\left(\phi^{-1}(a_{1})\otimes\cdots\otimes\phi^{-1}(a_{n})\right).
\]
So we may consider the $\widetilde{G}$-invariant spaces $\Hom(A^{\otimes n},A)^{\widetilde{G}}$.
It is easy to see that under the above embedding $C_{n}^{G}(A)$ are
contained in the $\widetilde{G}$-invariant spaces, making their dimensions
upper bounds for $c_{n}^{G}(A)$. 

In some special cases it is possible to replace the spaces $\Hom(A^{\otimes n},A)$
and $\Hom(A^{\otimes n},A)^{\widetilde{G}}$ by the, somewhat simpler,
spaces $T_{n+1}=\Hom(A^{\otimes(n+1)},F)$ and $T_{n+1}^{\widetilde{G}}=\Hom(A^{\otimes(n+1)},F)^{\widetilde{G}}$
(notice that $\widetilde{G}$ acts also on $T_{n}$ by 
\[
\phi(f)(a_{1}\otimes\cdots\otimes a_{n})=f\left(\phi^{-1}(a_{1})\otimes\cdots\otimes\phi^{-1}(a_{n})\right),
\]
 so the latter space makes sense). This is the case in the following
lemma:
\begin{lem}
\label{lem:trace}Let $A$ be a finite dimensional $G$-graded algebra.
Suppose that there exist a $\widetilde{G}$-invariant nondegenerate
bilinear form $tr(\,,\,)$ over $A$. Then 
\[
\psi:\Hom(A^{\otimes n},A)\rightarrow T_{n+1}
\]
given by $\psi(f)=tr(f,x_{n+1})$, where 
\[
tr(f,x_{n+1}):a_{1}\otimes\cdots\otimes a_{n+1}\mapsto tr(f(a_{1}\otimes\cdots\otimes a_{n}),a_{n+1})
\]
is $\widetilde{G}$-isomorphism. In particular $\psi$ identifies
$\Hom(A^{\otimes n},A)^{\widetilde{G}}$ with $T_{n+1}^{\widetilde{G}}$
. \end{lem}
\begin{proof}
$\psi$ is indeed a $\widetilde{G}$-homomorphism since 
\begin{eqnarray*}
\psi\left(g\left(f\right)\right)(a_{1}\otimes\cdots\otimes a_{n+1}) & = & tr(gf(g^{-1}a_{1}\otimes\cdots\otimes g^{-1}a_{n}),a_{n+1})\\
 & = & tr(f(g^{-1}a_{1}\otimes\cdots\otimes g^{-1}a_{n}),g^{-1}a_{n+1})=g\psi(f)
\end{eqnarray*}
Moreover, due to the nondegeneracy of $tr(\,,\,)$, $\psi$ is injective.
Finally, since $T_{n+1}$ and $\Hom(A^{\otimes n},A)$ have the same
dimension we also get surjectivity.
\end{proof}
In this paper we are concerned only with algebras $A$ which are $G$-graded
simple (and finite dimensional). Throughout the article we focus on
different families of such algebras and for each such family we will
define a functional $tr$ as required in the above lemma. 

The above discussion suggests that we should try to approximate $c_{n}^{G}(A)$
by $t_{n+1}^{G}(A)$. More precisely, we should prove that $t_{n+1}^{G}(A)=\dim_{F}T_{n}^{\widetilde{G}}(A)\sim c_{n}^{G}(A)$
and calculate $t_{n+1}^{G}(A)$ asymptotically. This will be done
for $G$-simple, finite dimensional algebras $A$.

\section{\label{sec:Review-of-simple} Finite dimensional $G$-graded algebras}

Let us briefly recall some of the basic definitions and results concerning
group gradings of (associative) algebras, and $G$-simple algebras.

Let $W$ be an associative algebra over a field $F$ and $G$ a finite
group. We say that $W$ is $G$-graded if there is an $F$-vector
space decomposition
\[
W=\bigoplus_{g\in G}W_{g}
\]
 where $W_{g}W_{h}\subseteq W_{gh}$.

In order to put the results of this work into the right perspective,
we introduce the following structure theory of $G$-graded finite
dimensional algebras. 

Suppose now that $W=A$ is finite dimensional semisimple over $F$,
where $F$ is algebraically closed of characteristic zero, and $J(A)$
is its Jacobson radical. The algebra $\overline{A}=A/J(A)$ is semisimple.
Since $F$ is algebraically closed, $\overline{A}$ is a direct sum
of matrix algebras over $F$. Furthermore, from Wedderburn and Malcev's
well known theorem one can lift idempotents from $\bar{A}$ to $A$
and obtain a ($F$-vector spaces) decomposition 
\[
A=\hat{A}\oplus J(A)
\]
where $\overline{A}\cong\hat{A}$ and $\hat{A}$ is an $F$-subalgebra
of $A$. In other words, there exists a semisimple subalgebra of $A$
which supplements $J(A)$ as an $F$-vector space. Now, if $A$ is
$G$-graded, the Wedderburn-Malcev decomposition is compatible with
the grading. Namely $J(A)$ is $G$-graded and one can find a semisimple
supplement which is $G$-graded as well.

Let us restrict ourself to the case where the algebra $A$ is semisimple.
In that case one knows that $A$ decomposes into a direct sum of semisimple,
$G$-graded simple algebras and so we further restrict and assume
that $A$ is a $G$-simple algebra. The $G$-graded structure of a
$G$-simple algebra was studied in \cite{key-5}. Before recalling
the precise result let us present two examples of $G$-simple algebras
which turn out to be typical.

The first is the well known group algebra $FG$ with the natural $G$-grading.
Clearly, it has no nontrivial $G$-graded two sided ideals since any
nonzero homogeneous element is invertible. More generally, one may
consider a twisted group algebra $F^{\alpha}G$ with the natural $G$-grading
(here $\alpha$ is a 2-cocycle of $G$ with values in $F^{\times}$).
Clearly it is $G$-simple. In fact, one may consider a somewhat more
general $G$-simple algebra: a twisted group algebra $A=F^{\alpha}H$,
$H<G$, by setting $A_{g}=0$ for $g\in G-H$. Such a grading, in
which each homogeneous component has dimension $\leq1$ is called
fine.

A different way to obtain a $G$-simple algebra is by grading $M_{m}(F)$
where each elementary matrix $e_{i,j}$ is homogeneous. Note that
since $M_{m}(F)$ is a simple algebra, any such grading determines
a $G$-simple algebra. These gradings are called elementary and may
be described in the following way. Choose an $m$-tuple $\mathfrak{g}=(\gamma_{1},...,\gamma_{m})\in G^{m}$.
Then we let the elementary matrix $e_{i,j}$ be homogeneous of degree
$g_{i}^{-1}g_{j}$. One can easily see that this determines a $G$-grading
on $M_{m}(F)$. A theorem of Bahturin, Sehgal and Zaicev claims that
any $G$-grading of a $G$-simple algebra over an algebraically closed
field $F$ of characteristic zero is a suitable combination of a fine
and elementary grading. Here is the precise statement.
\begin{thm}
\label{thm:Let--be}\cite{key-5}Let $A$ be a finite dimensional
$G$-simple algebra. Then there exists a subgroup $H$ of $G$, a
2-cocycle $\mu\in H^{2}(H,F^{\times})$ where the action of $H$ on
$F$ is trivial, an integer $m$ and an $m$-tuple $(\gamma_{1}=e,...,\gamma_{m})\in G^{m}$
such that $A$ is $G$-graded isomorphic to $C=F^{\mu}H\otimes M_{m}(F)$
where $C_{g}=\sp\left\{ b_{h}\otimes e_{i,j}:g=\gamma_{i}^{-1}h\gamma_{j}\right\} $.
Here $b_{h}\in F^{\mu}H$ is a representative of $h\in H$ and $e_{i,j}\in M_{m}(F)$
is the $(i,j)$ elementary matrix. In particular the idempotents $1\otimes e_{i,i}$
as well as the identity element of $A$ are homogeneous of degree
$e\in G$. 
\end{thm}

\section{\label{sec:Invariants}Invariants of $m\times m$ matrices with elementary
$G$-grading}

Let $A=M_{m}(F)$ graded by an elementary $G$-grading corresponding
to a vector $\mathfrak{g}\in G^{m}$. As in \secref{Motivation} we
denote by $\widetilde{G}$ the group $\Aut_{G}(A)$ - automorphisms
which preserve the $G$-grading on $A$, and by $T_{n}$ the $\widetilde{G}$-module
$\Hom(A^{\otimes n},F)$. Here the action is given by 
\[
\phi\cdot f\left(a_{1}\otimes\cdots\otimes a_{n}\right)=f\left(\phi^{-1}(a_{1})\otimes\cdots\otimes\phi^{-1}(a_{n})\right)
\]
where $\phi\in\widetilde{G}$ and $f\in T_{n}$.

\subsection{Calculation of $T_{n}^{\widetilde{G}}$}

Before we can calculate $T_{n}^{\widetilde{G}}$ and show that $t_{n+1}^{G}(A)\sim c_{n}^{G}(A)$,
we should have a better understanding of $\widetilde{G}$ and $T_{n}^{\widetilde{G}}$.
This will be accomplished by obtaining a more ``concrete'' presentation
of these objects.

We may assume that $\mathfrak{g}=(\gamma_{1},..,\gamma_{1},...,\gamma_{k},...,\gamma_{k})$,
where the $\gamma_{i}$ are distinct and $\gamma_{i}$ appears $m_{i}=m_{\gamma_{i}}>0$
times (we therefore use the notation $\mathfrak{g}=(\gamma_{1}^{m_{1}},...,\gamma_{k}^{m_{k}})$).
The set $\{\gamma_{1},...,\gamma_{k}\}$ of the different group elements
appearing in the grading vector $\mathfrak{g}$ is denoted by $B$.
Next, since matrices are involved, we need a notation for writing
matrices which is compatible with the grading. Therefore, we are indexing
the rows and columns of any $m\times m$ matrix by $\gamma_{1}(1),...,\gamma_{1}(m_{1}),...,\gamma_{k}(1),...,\gamma_{k}(m_{k})$. 
\begin{rem}
We should have used the notation $(\gamma_{1},1),...,(\gamma_{k},m_{k})$
instead, however we find it less readable.
\end{rem}
Moreover, if $X=(x_{\gamma_{i}(r_{i}),\gamma_{j}(r_{j})})$ is an
$m\times m$ matrix, \emph{the $(\gamma_{i},\gamma_{j})$}\textit{-block}\emph{
of $X$ }is 
\[
\widehat{X}_{\gamma_{i},\gamma_{j}}=\sum_{r_{i}=1}^{m_{g_{i}}}\sum_{r_{j}=1}^{m_{g_{j}}}x_{\gamma_{i}(r_{i}),\gamma_{j}(r_{j})}e_{r_{i},r_{j}}\in M_{m_{\gamma_{i}}\times m_{\gamma_{j}}}(F)
\]
 where $e_{r_{i},r_{j}}$ is an elementary matrix. So, 
\[
X=\left(\begin{array}{ccc}
\widehat{X}_{\gamma_{1},\gamma_{1}} & \cdots & \widehat{X}_{\gamma_{1},\gamma_{k}}\\
\vdots & \ddots & \vdots\\
\widehat{X}_{\gamma_{k},\gamma_{1}} & \cdots & \widehat{X}_{\gamma_{k},\gamma_{k}}
\end{array}\right)
\]
 Denote
\[
X_{\gamma_{i},\gamma_{j}}=\sum_{r_{i}=1}^{m_{\gamma_{i}}}\sum_{r_{j}=1}^{m_{\gamma_{j}}}x_{\gamma_{i}(r_{i}),\gamma_{j}(r_{j})}e_{\gamma_{i}(r_{i}),\gamma_{j}(r_{j})}
\]
i.e. the $(\gamma_{i},\gamma_{j})$-block of $X_{\gamma_{i},\gamma_{j}}$
is $\widehat{X}_{\gamma_{i},\gamma_{j}}$ and all other blocks are
zero. Note that $X_{\gamma_{i},\gamma_{j}}\in A_{\gamma_{i}^{-1}\gamma_{j}}$.
Finally, matrices (in $A$) of the form $X_{\gamma_{i},\gamma_{j}}$
(that is, all the blocks apart from the $(\gamma_{i},\gamma_{j})$-block
are zero) will be called $(\gamma_{i},\gamma_{j})$\emph{-matrices}. 
\begin{rem}
\emph{The notation introduced here will be used freely throughout
most of the paper. }
\end{rem}

\begin{rem}
We distinguish between the matrices $X_{\gamma_{i},\gamma_{j}}$ and
$\widehat{X}_{\gamma_{i},\gamma_{j}}$ only for the sake of clarity
in some of the proofs appearing in this section. 
\end{rem}
By Noether-Skolem theorem, the automorphism group of $A$ (denoted
by $\Aut(A)$) consists of all transformation $\tau_{X}:A\to A$ given
by $\tau_{X}(a)=X^{-1}aX$, where $a\in A$ and $X\in GL_{m}(F)$.
Therefore, we got an epimorphism of groups $GL_{m}(F)\to\Aut(A)$.
The kernel is well known and is equal to the center of $GL_{m}(F)$,
which is the subgroup of non-zero scalar matrices. Therefore, $\Aut(A)$
is (isomorphic to) $PGL_{m}(F)$. As a result, in order to understand
the group $\widetilde{G}$ it is enough to find all matrices $X\in GL_{m}(F)$
such that $\tau_{X}$ preserves the $G$-graded structure. In other
words, we intend to describe the group, denoted by $\overline{G}$,
consisting of all matrices $X\in GL_{m}(F)$ such that $X^{-1}A_{g}X\subseteq A_{g}$
for every $g\in G$.

Let us start with a lemma:
\selectlanguage{american}%
\begin{lem}
\label{lem:zero}Let $X\in GL_{m}(F)$. Suppose that $X^{-1}A_{g}X\subseteq A_{g}$
for every $g\in G$. Then for every $t\in B$ there is a unique $h\in B$
such that $m_{t}=m_{h}$, $\widehat{X}_{t,h}$ is non-singular and
$\widehat{X}_{t,s}=0$ for every $s\ne h$.\end{lem}
\begin{proof}
Let $X=\sum_{s,t\in B}X_{s,t}$. First note that $\phi:A_{e}\rightarrow A_{e}$;
$\phi(Y)=X^{-1}YX$ is an automorphism of $F$-algebras. Since $A_{e}=\prod_{t\in B}M_{m_{t}}(F)$,
a direct product of simple $F$-algebras, every ideal of $A_{e}$
must be of the form \foreignlanguage{english}{$\prod_{t\in B}J_{t}$,
where $J_{t}\vartriangleleft M_{m_{t}}(F)$.} Therefore, \foreignlanguage{english}{we
may regard each $M_{m_{t}}(F)$ as a minimal ideal of $A_{e}$. Because
automorphisms take minimal ideals to minimal ideals, we obtain tha}t
$\phi$ must take each $M_{m_{t}}(F)$ to some distinct $M_{m_{h}}(F)$.
Moreover, because automorphisms preserve dimension, $m_{t}=m_{h}$.

\selectlanguage{english}%
Hence, if $Y=Y_{t_{0},t_{0}}$ and $\widehat{Y}_{t_{0},t_{0}}=I_{m_{t_{0}}}$,
there is an $h_{0}\in B$, such that $X^{-1}YX$ is the $(h_{0},h_{0})$-matrix
$Z=Z_{h_{0},h_{0}}$, where $\hat{Z}=I_{m_{h_{0}}}$ (automorphism
sends central idempotent to central idempotent) . Thus, 
\[
XZ=YX\Rightarrow\sum_{s\in B}X_{s,h_{0}}=\sum_{r\in B}X_{t_{0},r}.
\]
So $\widehat{X}_{s,h_{0}}=0$ for $s\ne t_{0}$. Moreover, if $\widehat{X}_{t_{0},h_{0}}$
is singular (i.e. its columns are linearly dependent), then the $m_{h_{0}}$
columns of $X$ corresponding to $h_{0}$ are also linearly dependent.
This is impossible since $X\in GL_{m}(F)$. 
\end{proof}
\selectlanguage{english}%
Therefore, if $X\in\overline{G}$ it must be of the form: 
\[
X=\sum_{t\in B}X_{t,h_{t}},
\]
where $h:B\to B$ is a bijection, such that $m_{t}=m_{h_{t}}$. In
other words, $X$ is a block (generalized) diagonal regular matrix
of $GL_{m}(F)$. This simple structure makes the computation of the
inverse $X^{-1}$ very simple. This will be reflected in the proof
of the next lemma. 
\begin{lem}
\label{lem:With-the-above}With the above notations. There is some
$g\in G$ such that $h_{t}=gt$.
\selectlanguage{american}%
\begin{proof}
Note that $X^{-1}=\sum_{t\in B}\left(X^{-1}\right)_{t,h_{t}}$, where
\foreignlanguage{english}{$\widehat{X^{-1}}{}_{t,h_{t}}=\left(\widehat{X}_{h_{t},t}\right)^{-1}$}.
Therefore,
\[
X^{-1}=\sum_{t\in B}\left(\left(\widehat{X}_{h_{t},t}\right)^{-1}1_{t}(\in A_{h_{t}^{-1}t})\right)
\]
(here $1_{t}=\sum_{i}e_{t(i),t(i)}$.) 

Fix $r\in B$ and a nonzero $(r,t)$-matrix $Y=Y_{r,t}\in A_{r^{-1}t}$.
We have: 
\[
0\neq X^{-1}YX=\left(\widehat{X}_{h_{t},t}\right)^{-1}1_{t}Y_{r,t}X_{t,h_{t}}=\in A_{h{}_{r}^{-1}h_{t}}.
\]
Since $X\in\overline{G}$ the matrix $X^{-1}YX$ must be in the same
$G$-graded component as $Y$. That is, in $A_{t^{-1}r}$. As a result,
$h{}_{r}^{-1}h_{t}=r^{-1}t$, which implies that $th{}_{t}^{-1}=rh{}_{r}^{-1}=:g^{-1}$.
Thus, $h_{t}=gt$ for every $t\in B$.
\end{proof}
\end{lem}
The element $g$ which is associated to the function $h:B\to B$ has
two properties: $gB=B$ and $m_{t}=m_{gt}$. This leads to the following
definition.
\selectlanguage{american}%
\begin{defn}
\label{def:sub_groups_of_G}$H_{B}$ is the subgroup\textbf{ }of $G$
consisting of the elements $g$ such that $gB=B$. The $\mathfrak{g}$-subgroup
of $G$ is $H_{\mathfrak{g}}=\left\{ g\in H_{B}\,|\, m_{gt}=m_{t}\mbox{ for every \ensuremath{g\in B}}\right\} $.\end{defn}
\selectlanguage{english}%
\begin{prop}
The subgroup $\overline{G}$ of $GL_{m}(F)$ is equal to

\textup{
\[
\left\{ \sum_{t\in B}X_{t,gt}\in GL_{m}(F)\,|\, g\in H_{\mathfrak{g}}\right\} .
\]
In other words, $\overline{G}$ is equal to the semi-direct product
$\left(GL_{m_{1}}(F)\times\cdots\times GL_{m_{k}}(F)\right)\ltimes H_{\mathfrak{g}}$.
Therefore, the group $\widetilde{G}$ is the semi-direct product $\left(\frac{GL_{m_{1}}(F)\times\cdots\times GL_{m_{k}}(F)}{F^{\times}I_{m}}\right)\ltimes H_{\mathfrak{g}}$.}\end{prop}
\begin{proof}
Let $GL_{m}(F)\ni X=\sum_{t\in B}X_{t,gt}$, $g\in H_{\mathfrak{g}}$.
It is obvious that $\widehat{X}_{t,tg}$ are invertible matrices for
all $t\in B$. For a $(t,s)$-matrix $Y=Y_{t,s}\in A_{t^{-1}s}$:
\[
X^{-1}Y_{t,s}X=\left(\widehat{X}_{gt,t}\right)^{-1}1_{t}Y_{t,s}X{}_{s,gs}\in A_{(gt)^{-1}gs=t^{-1}s}
\]
since $X^{-1}=\sum_{s\in B}\left(\widehat{X}_{gt,t}\right)^{-1}1_{t}$.
Hence, \textbf{$X\in\widetilde{G}$}.\textbf{ }

The other direction is \lemref{With-the-above}.\end{proof}
\begin{example}
\label{exa:three_extreme}Let us look on three extremal examples of
gradings on $A$ together with the resulting groups $\widetilde{G}.$
\begin{enumerate}
\item If $G$ is the trivial group (or the grading vector $\mathfrak{g}=(g,....,g)=(g^{m_{1=k}})$
for some $g\in G$), we obtain $H_{\mathfrak{g}}=\{e\}$ and conclude
that $\widetilde{G}=PGL_{m}(F)$. 
\item If $m=2$ and $G=\mathbb{Z}_{2}=\{0,1\}$ one can consider the grading
corresponding to the grading vector $\mathfrak{g}=(0,1)$. In this
case, $A=\begin{pmatrix}0 & 1\\
1 & 0
\end{pmatrix}$ (that is, $A_{0}=Fe_{0,0}\oplus Fe_{1,1}$ and $A_{1}=Fe_{1,0}\oplus Fe_{1,0}$)
and $\widetilde{G}=\left(\frac{F^{\times}\times\cdots\times F^{\times}}{F^{\times}I_{m}}\right)\ltimes G$. 
\item Finally, consider the case $m=2;G=\mathbb{Z}_{3}=\{0,1,2\}$ and $\mathfrak{g}=(0,1)$.
In this case $A=\begin{pmatrix}0 & 2\\
1 & 0
\end{pmatrix}$ and $\widetilde{G}=\frac{F^{\times}\times\cdots\times F^{\times}}{F^{\times}I_{m}}$. 
\end{enumerate}
\end{example}
\selectlanguage{american}%
Let $V$ \foreignlanguage{english}{be an $m$-dimensional $F$-vector
space. }Recall from the classical theory of invariants (e.g. see \cite{key-Procesi})
that the group $GL_{m}(F)$ embeds into $\End_{F}(V^{\otimes n})$
via the diagonal action 
\[
P(v_{1}\otimes\cdots\otimes v_{n})=Pv_{1}\otimes\cdots\otimes Pv_{n}
\]
 on $V^{\otimes n}$. We can therefore define an action of $GL_{m}(F)=GL(V)$
on $\End_{F}(V^{\otimes n})$ via conjugation: 
\[
X\cdot\phi=X^{-1}\circ\phi\circ X,
\]
where $X\in GL(V)$ and $\phi\in\End_{F}(V^{\otimes n})$. Turning
$\End_{F}(V^{\otimes n})$ to a $GL_{m}(F)$-module. Furthermore,
this module is isomorphic to the $GL_{m}(F)$-module $T_{n}$ (we
will discuss this isomorphism explicitly in \secref{Hilbert-series}).
Therefore, the space \foreignlanguage{english}{$T_{n}^{\widetilde{G}}=T_{n}^{\overline{G}}$
can be identified with the $F$-space $C_{\overline{G}}\left(\End_{F}(V^{\otimes n})\right)$,
}the centralizer of $\overline{G}$ in $\End_{F}(V^{\otimes n})$
i.e. \foreignlanguage{english}{
\[
T_{n}^{\widetilde{G}}\cong C_{\overline{G}}\left(\End_{F}(V^{\otimes n})\right)
\]
 (as $F$-spaces). }

\selectlanguage{english}%
\uline{Our next step} is to find explicitly a ``nice'' spanning
set for $C_{\overline{G}}\left(\End_{F}(V^{\otimes n})\right)$. In
other words, we will state and prove the first fundamental theorem
of $T_{n}$ with respect to the group $\widetilde{G}$. 

\selectlanguage{american}%
We equip the vector space $V$ with a $G$-grading $V=\oplus_{g\in B}V_{g}$,
where $V_{g}=\sp\left\{ e_{g(j)}|j=1,...,m_{g}\right\} $\foreignlanguage{english}{.}

\selectlanguage{english}%
For every $\sigma\in S_{n}$ and $\mathbf{h}\in B^{n}$ denote by
$T{}_{\sigma,\mathbf{h}}^{\prime}$ the linear operator in $\End_{F}(V^{\otimes n})$
given by\foreignlanguage{american}{
\[
T{}_{\sigma,\mathbf{h}}^{\prime}\left(v_{t_{1}}\otimes\cdots\otimes v_{t_{n}}\right)=\begin{cases}
v_{t_{\sigma(1)}}\otimes\cdots\otimes v_{t_{\sigma(n)}} & \mbox{if \ensuremath{(t_{1},...,t_{n})=\mathbf{h}}}\\
0 & \mbox{otherwise}
\end{cases},
\]
where $v_{t_{i}}\in V_{t_{i}}$. }
\selectlanguage{american}%
\begin{thm}
\label{thm:centralizer} The centralizer of $\overline{G}$ in $\End_{F}(V^{\otimes n})$
is the $F$-span of 
\[
T_{\sigma,\mathbf{h}}=\sum_{h\in H_{\mathfrak{g}}}T{}_{\sigma,h\mathbf{h}}^{\prime}
\]
where $\mathbf{h}=(h_{1},...,h_{n})\in B^{n}$ and $\sigma\in S_{n}$.\end{thm}
\selectlanguage{english}%
\begin{example}
\label{exa:invariants_example}Let us write explicitly $C_{\overline{G}}\left(\End_{F}(V^{\otimes n})\right)$
(assuming \thmref{centralizer}) for the three cases in \exaref{three_extreme}.
\begin{enumerate}
\item In this case $|B|=1$, so there is only one vector in $B^{n}$. Hence,
in this case we may write $T_{\sigma}$ instead of $T{}_{\sigma,\mathbf{h}}^{\prime}$
and $T{}_{\sigma,\mathbf{h}}$. Moreover, $T{}_{\sigma}\left(v_{1}\otimes\cdots\otimes v_{n}\right)=v_{\sigma(1)}\otimes\cdots\otimes v_{\sigma(n)}$
for every $v_{1},...,v_{n}\in V$. This description of $C_{\overline{G}}\left(\End_{F}(V^{\otimes n})\right)$
is well known (see \cite{key-Procesi}).
\item Here we get that $V=V_{0}\oplus V_{1}$, where $V_{0}=Fe_{0},V_{1}=Fe_{1}$.
Therefore, for $\mathbf{h}=(0,1,0)$ and $\sigma=(1\,2\,3)$ we have
$T{}_{\sigma,\mathbf{h}}^{\prime}(e_{0}\otimes e_{1}\otimes e_{0})=e_{1}\otimes e_{0}\otimes e_{0}$.
For every other tensor which do not match $\mathbf{h}$, the endomorphism
$T{}_{\sigma,\mathbf{h}}^{\prime}$ returns zero. Since $H_{\mathfrak{g}}=G=\mathbb{Z}_{2}$,
\[
T_{\sigma,\mathbf{h}}(e_{0}\otimes e_{1}\otimes e_{0})=e_{1}\otimes e_{0}\otimes e_{0};\,\,\, T_{\sigma,\mathbf{h}}(e_{1}\otimes e_{0}\otimes e_{1})=e_{0}\otimes e_{1}\otimes e_{1}.
\]

\item In this case, as in the previous case, $V=V_{0}\oplus V_{1}$, where
$V_{0}=Fe_{0},V_{1}=Fe_{1}$. Take $\mathbf{h}=(0,1,0)$ and $\sigma=(1\,2\,3)$.
Then, $T{}_{\sigma,\mathbf{h}}^{\prime}(e_{1}\otimes e_{0}\otimes e_{0})=e_{1}\otimes e_{0}\otimes e_{0}$.
However, since $H_{\mathfrak{g}}=\{e\}$, we get that 
\[
T{}_{\sigma,\mathbf{h}}=T{}_{\sigma,\mathbf{h}}^{\prime}.
\]

\end{enumerate}
\end{example}
\selectlanguage{american}%
Before we can tackle this theorem we need to do some preparations.
Let us decompose the space $V^{\otimes n}$ in the following manner:
\[
V^{\otimes n}=\bigoplus_{\mathbf{h}\in B^{n}}\left(V_{\mathbf{h}}=V_{h_{1}}\otimes\cdots\otimes V_{h_{n}}\right).
\]
It is clear that for every $\mathbf{h}\in B^{n}$, $V_{\mathbf{h}}$
is a sub-representation of the group 
\[
\overline{G^{\prime}}=GL(V_{\gamma_{1}})\times\cdots\times GL(V_{\gamma_{k}})=GL_{m_{1}}(F)\times\cdots\times GL_{m_{k}}(F)<\overline{G.}
\]
 Here $X=X_{\gamma_{1},\gamma_{1}}+\cdots+X_{\gamma_{k},\gamma_{k}}\in\overline{G^{\prime}}$
maps $v_{h_{1}}\otimes\cdots\otimes v_{h_{n}}\in\End_{F}(V_{\mathbf{h}})$
to $X_{h_{1},h_{1}}v_{h_{1}}\otimes\cdots\otimes X_{h_{n},h_{n}}v_{h_{n}}$.

To proceed we introduce the following notation: The image of the group
algebra $F\overline{G^{\prime}}$ inside $\End_{F}(V_{\mathbf{h}})$
is denoted by $\left\langle F\overline{G^{\prime}}\right\rangle _{\mathbf{h}}$.
Denote the sub-group of $S_{n}$ consisting from all the permutations
$\sigma$ such that $\mathbf{h}=\mathbf{h}^{\sigma}:=(h_{\sigma(1)},...,h_{\sigma(n)})$
by $S_{\mathbf{h}}$. Clearly $S_{\mathbf{h}}$ acts on $\End_{F}(V_{\mathbf{h}})$
by 
\[
\sigma(v_{1}\otimes\cdots\otimes v_{n})=v_{\sigma^{-1}(1)}\otimes\cdots\otimes v_{\sigma^{-1}(n)}.
\]
 Denote by $\left\langle FS_{\mathbf{h}}\right\rangle $ the image
of the group algebra $FS_{\mathbf{h}}$ inside $\End_{F}(V_{\mathbf{h}})$.
We will use the double centralizer theorem to deduce $C\left(\left\langle F\overline{G^{\prime}}\right\rangle _{\mathbf{h}}\right)=\left\langle FS_{\mathbf{h}}\right\rangle $,
by first proving:
\begin{prop}
\label{prop:preparations_list_of_invariants}$\left\langle F\overline{G^{\prime}}\right\rangle _{\mathbf{h}}=C\left(\left\langle FS_{\mathbf{h}}\right\rangle \right)$
(notice that $\left\langle F\overline{G^{\prime}}\right\rangle _{\mathbf{\mathbf{h}}}$
and $\left\langle FS_{\mathbf{h}}\right\rangle $ are both semisimple
subalgebras of the simple algebra $\End(V_{\mathbf{h}})$).\end{prop}
\selectlanguage{english}%
\begin{proof}
Consider the natural isomorphism $\Phi:\End_{F}(V_{h_{1}})\otimes\cdots\otimes\End_{F}(V_{h_{n}})\to\End_{F}(V_{\mathbf{h}})$
given by 
\[
\Phi(P_{1}\otimes\cdots\otimes P_{n})(v_{1}\otimes\cdots\otimes v_{n})=P_{1}(v_{1})\otimes\cdots\otimes P_{n}(v_{n}).
\]
(This is indeed an isomorphism since it is onto and the dimensions
of the source and the range are equal.) $S_{\mathbf{t}}$ acts on
$\End_{F}(V_{h_{1}})\otimes\cdots\otimes\End_{F}(V_{h_{n}})$ by 
\[
\sigma(P_{1}\otimes\cdots\otimes P_{n})=P_{\sigma^{-1}(1)}\otimes\cdots\otimes P_{\sigma^{-1}(n)}.
\]
Moreover, we claim that 
\[
\Phi(\sigma\cdot P_{1}\otimes\cdots\otimes P_{n})=\sigma\Phi(P_{1}\otimes\cdots\otimes P_{n})\sigma^{-1}.
\]
Indeed,
\begin{eqnarray*}
\sigma\Phi(P_{1}\otimes\cdots\otimes P_{n})\sigma^{-1}\left(v_{1}\otimes\cdots\otimes v_{n}\right) & = & \sigma\Phi(P_{1}\otimes\cdots\otimes P_{n})(v_{\sigma(1)}\otimes\cdots\otimes v_{\sigma(n)})\\
 & = & \sigma\left(P_{1}v_{\sigma(1)}\otimes\cdots\otimes P_{n}v_{\sigma(n)}\right)\\
 & = & P_{\sigma^{-1}(1)}v_{1}\otimes\cdots\otimes P_{\sigma^{-1}(n)}v_{n}\\
 & = & \Phi\left(\sigma\cdot(P_{1}\otimes\cdots\otimes P_{n}\right)(v_{1}\otimes\cdots\otimes v_{n}).
\end{eqnarray*}
Therefore, in order to find $C\left(\left\langle FS_{\mathbf{\mathbf{h}}}\right\rangle \right)$
we should find the elements in $\End_{F}(V_{h_{1}})\otimes\cdots\otimes\End_{F}(V_{h_{n}})$
which are stable under the action of $S_{\mathbf{h}}$. It is clear
that for $X=X_{\gamma_{1},\gamma_{1}}+\cdots+X_{\gamma_{k},\gamma_{k}}\in\overline{G^{\prime}}$:
\[
\Phi(\widehat{X}_{h_{1},h_{1}}\otimes\cdots\otimes\widehat{X}_{h_{n},h_{n}})=X
\]
(here we really mean: the image of $X$ inside $\End_{F}(V_{\mathbf{\mathbf{h}}})$.)
Thus, we are left with proving that the only elements of $\End_{F}(V_{h_{1}})\otimes\cdots\otimes\End_{F}(V_{h_{n}})$
which are stable under the action of $S_{\mathbf{t}}$ are the ones
spanned by $\widehat{X}_{h_{1},h_{1}}\otimes\cdots\otimes\widehat{X}_{h_{n},h_{n}}$. 

Since $S_{\mathbf{h}}$ is a finite group it is clear that the stable
vectors are exactly 
\[
\sum_{s\in S_{\mathbf{h}}}s(u),
\]
where $u$ is an element of $\End_{F}(V_{h_{1}})\otimes\cdots\otimes\End_{F}(V_{h_{n}})$.
Therefore, the proposition follows from the next lemma applied to
$Y_{\gamma_{1}}=GL_{m_{1}}(F),...,Y_{\gamma_{k}}=GL_{m_{k}}(F)$ and
$W_{\gamma_{1}}=\End_{F}(V_{\gamma_{1}}),...,W_{\gamma_{k}}=\End_{F}(V_{\gamma_{k}})$.\end{proof}
\begin{lem*}
Let $W=W_{\gamma_{1}}\oplus\cdots\oplus W_{\gamma_{k}}$ be a finite
dimensional vector space over $F$ and $Y=Y_{\gamma_{1}}\oplus\cdots\oplus Y_{\gamma_{k}}\subseteq W$
such that $FY_{\gamma_{i}}=W_{\gamma_{i}}$, then $\left(W_{h_{1}}\otimes\cdots\otimes W_{h_{n}}\right)^{S_{\mathbf{t}}}$
is spanned by $y_{h_{1}}\otimes\cdots\otimes y_{h_{n}}$, where $y_{h_{1}}\in Y_{h_{1}},...,y_{h_{n}}\in Y_{h_{n}}$.\end{lem*}
\begin{proof}
For the sake of minimizing the notation we assume without the loss
of generality: 
\[
\gamma_{1}=h_{1}=\cdots=h_{n_{1}};\,...;\,\gamma_{k}=h_{n_{1}+\cdots+n_{k-1}+1}=\cdots=h_{n=n_{1}+\cdots+n_{k}}.
\]
 So we can write $W_{\gamma_{1}}^{\otimes n_{1}}\otimes\cdots\otimes W_{\gamma_{k}}^{\otimes n_{k}}$
and $S_{\mathbf{\mathbf{h}}}=S_{n_{1}}\times\cdots\times S_{n_{k}}$. 

Let $w_{1,i},...,w_{d_{i},i}$ be a basis of $W_{\gamma_{i}}$. Each
$S_{\mathbf{h}}$-orbit of an element from $W_{\gamma_{1}}^{\otimes n_{1}}\otimes\cdots\otimes W_{\gamma_{k}}^{\otimes n_{k}}$
contains a unique element of the form: 
\[
w_{1,1}^{\otimes r_{1,1}}\otimes\cdots\otimes w_{d_{1},1}^{\otimes r_{d_{1},1}}\otimes\cdots\otimes w_{1,k}^{\otimes r_{1,k}}\otimes\cdots\otimes w_{d_{k},k}^{\otimes r_{d_{k},k}},
\]
where $r_{i,1}+\cdots+r_{i,d_{i}}=n_{i}$. Denote by $\mathbf{w}_{r_{1,1},...,r_{d_{k},k}}$
the sum of all the elements in the orbit of the above element. It
is clear that the $\mathbf{w}_{r_{1,1},...,r_{d_{k},k}}$ are a basis
for $\left(W_{\gamma_{1}}^{\otimes n_{1}}\otimes\cdots\otimes W_{\gamma_{k}}^{\otimes n_{k}}\right)^{S_{\mathbf{t}}}$.
Suppose $f:\left(W_{\gamma_{1}}^{\otimes n_{1}}\otimes\cdots\otimes W_{\gamma_{k}}^{\otimes n_{k}}\right)^{S_{\mathbf{t}}}\to F$
is a linear functional which vanishes on $y_{\gamma_{1}}^{\otimes n_{1}}\otimes\cdots\otimes y_{t_{k}}^{\gamma_{k}}$,
where $y_{\gamma_{1}}\in Y_{\gamma_{1}},...,y_{\gamma_{k}}\in Y_{\gamma_{k}}$.
To prove the lemma we need to show that $f=0$. 

Write $y_{\gamma_{i}}=\sum\alpha_{j.i}w_{j.i}$ for $i=1,...,k$.
Then,
\[
y_{\gamma_{1}}^{\otimes n_{1}}\otimes\cdots\otimes y_{t_{k}}^{\gamma_{k}}=\sum_{r_{1,1},...,r_{d_{k},k}}\left(\left(\prod_{i,j}\alpha_{j,i}^{r_{j,i}}\right)\mathbf{w}_{r_{1,1},...,r_{d_{k},k}}\right).
\]
Therefore, by applying $f$ on this equality we obtain: 
\[
0=\sum_{r_{1,1},...,r_{d_{k},k}}\left(\left(\prod_{i,j}\alpha_{j,i}^{r_{j,i}}\right)\cdot f\left(\mathbf{w}_{r_{1,1},...,r_{d_{k},k}}\right)\right).
\]
By assumption, this holds for every choice of $\alpha_{i,j}$, so
$f\left(\mathbf{w}_{r_{1,1},...,r_{d_{k},k}}\right)=0$ for all $r_{1,1},...,r_{d_{k},k}$. 
\end{proof}
Now we are ready to face \thmref{centralizer}:
\selectlanguage{american}%
\begin{proof}
First note that if $P_{g}=\sum_{t\in B}X_{gt,t}\in\overline{G}$ and
$v_{t}\in V_{t}$, then 
\[
P_{g}v_{t}=X_{gt,t}v_{t}\in V_{gt}.
\]
Therefore, for any $v_{t_{1}}\otimes\cdots\otimes v_{t_{n}}\in V_{t_{1}}\otimes\cdots\otimes V_{t_{n}}$where
$\mathbf{t}=(t_{1},...,t_{n})\in H_{\mathfrak{g}}\mathbf{h}$: 
\begin{eqnarray*}
T_{\sigma,\mathbf{h}}P_{g}\left(v_{t_{1}}\otimes\cdots\otimes v_{t_{n}}\right) & = & T_{\sigma,\mathbf{h}}\left(X_{gt_{1},t_{1}}v_{t_{1}}\otimes\cdots\otimes X_{gt_{n},t_{n}}v_{t_{n}}\right)\\
 & = & X_{gt_{\sigma(1)},t_{\sigma(1)}}v_{t_{\sigma(1)}}\otimes\cdots\otimes X_{gt_{\sigma(n)},t_{\sigma(n)}}v_{t_{\sigma(n)}}.
\end{eqnarray*}

Whereas 
\begin{eqnarray*}
P_{g}T_{\sigma,\mathbf{h}}\left(v_{t_{1}}\otimes\cdots\otimes v_{t_{n}}\right) & = & P_{g}\left(v_{t_{\sigma(1)}}\otimes\cdots\otimes v_{t_{\sigma(n)}}\right)\\
 & = & X_{gt_{\sigma(1)},t_{\sigma(1)}}v_{t_{\sigma(1)}}\otimes\cdots\otimes X_{gt_{\sigma(n)},t_{\sigma(n)}}v_{t_{\sigma(n)}}.
\end{eqnarray*}
Moreover, if $v_{t_{1}}\otimes\cdots\otimes v_{t_{n}}\in V_{t_{1}}\otimes\cdots\otimes V_{t_{n}}$,
where $\mathbf{t}=(t_{1},...,t_{n})\notin H_{\mathfrak{g}}\mathbf{h}$
, then 
\[
T_{\sigma,\mathbf{h}}P_{g}\left(v_{t_{1}}\otimes\cdots\otimes v_{t_{n}}\right)=T_{\sigma,\mathbf{h}}\left(X_{gt_{1},t_{1}}v_{t_{1}}\otimes\cdots\otimes X_{gt_{n},t_{n}}v_{t_{n}}\right)=0
\]
since $g\mathbf{t}\notin H_{\mathfrak{g}}\mathbf{h}$. Whereas 
\[
P_{g}T_{\sigma,\mathbf{h}}\left(v_{t_{1}}\otimes\cdots\otimes v_{t_{n}}\right)=P_{g}\left(0\right)=0.
\]
The first direction is proved (i.e. $\sp\{T_{\sigma,\mathbf{h}}|\sigma\in S_{n};\mathbf{h}\in B^{n}\}\subseteq C_{\overline{G}}\left(\End_{F}(V^{\otimes n})\right)$).

For the other direction let $0\ne T\in C_{\overline{G}}\left(\End_{F}(V^{\otimes n})\right)$.
Recall that the vector space $V^{\otimes n}$ decomposes into a direct
sum 
\[
V^{\otimes n}=\bigoplus_{\mathbf{t}\in B^{n}}\left(V_{\mathbf{t}}=V_{t_{1}}\otimes\cdots\otimes V_{t_{n}}\right).
\]
So, for $v_{\mathbf{h}}=v_{h_{1}}\otimes\cdots\otimes v_{h_{n}}\in V_{h_{1}}\otimes\cdots\otimes V_{h_{n}}$
we can write: 
\[
T(v_{\mathbf{h}})=\sum_{\mathbf{\mathbf{t}}\in B^{n}}\varphi_{\mathbf{t}}(v_{\mathbf{h}})
\]
where \foreignlanguage{english}{$\varphi_{\mathbf{t}}(v_{\mathbf{h}})\in V_{\mathbf{t}}$.}

Consider a diagonal matrix $X=\sum_{t\in B}X_{t,t}\in\overline{G^{\prime}}$,
where $\widehat{X}_{t,t}=\beta_{t}I_{m_{t}}$ ($\beta_{t}\in F$).
We compute:

\[
TX\left(v_{\mathbf{h}}\right)=T\left(\beta_{h_{1}}v_{h_{1}}\otimes\cdots\otimes\beta_{h_{n}}v_{h_{n}}\right)=\left(\prod_{i=1}^{n}\beta_{h_{i}}\right)\sum_{\mathbf{\mathbf{t}}\in B^{n}}\varphi_{\mathbf{t}}(v_{\mathbf{h}})
\]
\[
XT\left(v_{\mathbf{h}}\right)=X\left(\sum_{\mathbf{\mathbf{t}}\in B^{n}}\varphi_{\mathbf{t}}(v_{\mathbf{h}})\right)=\sum_{\mathbf{\mathbf{t}}\in B^{n}}\left(\prod_{i=1}^{n}\beta_{t_{i}}\right)\varphi_{\mathbf{t}}(v_{\mathbf{h}})
\]
Since $TX=XT$, by choosing suitable $\beta_{t}$'s we deduce that
$\varphi_{\mathbf{t}}(v_{\mathbf{h}})=0$ when $\mathbf{t}=(t_{1},..,t_{n})$
is not a permutation of $\mathbf{h}$. Hence: \foreignlanguage{english}{
\[
T\left(v_{\mathbf{h}}\right)=\sum_{\sigma\in S_{n}/\mathbf{h}}\varphi_{\mathbf{h}^{\sigma}}(v_{\mathbf{h}})
\]
where $\mathbf{h^{\sigma}=}(h_{\sigma(1)},...,h_{\sigma(n)})$ and
$S_{n}/\mathbf{\mathbf{h}}$ is a representative set corresponding
to the equivalence relation $\sim_{\mathbf{h}}$ defined by $\sigma\sim_{\mathbf{h}}\tau$
if and only if $\mathbf{h^{\sigma}}=\mathbf{h^{\tau}}$. }

\selectlanguage{english}%
Next, consider a block diagonal matrix $X=\sum_{t\in B}X_{t,t}\in\overline{G^{\prime}}$.
\foreignlanguage{american}{The equality $TX=XT$} implies that: 
\[
\sum_{\sigma\in S_{n}/\mathbf{h}}X\varphi_{\mathbf{h}^{\sigma}}(v_{\mathbf{h}})=XT(v_{\mathbf{h}})=T(Xv_{\mathbf{h}})=\sum_{\sigma\in S_{n}/\mathbf{h}}\varphi_{\mathbf{h^{\sigma}}}(Xv_{\mathbf{h}}).
\]
Since $X\varphi_{\mathbf{h}^{\sigma}}(v_{\mathbf{h}}),\varphi_{\mathbf{h^{\sigma}}}(Xv_{\mathbf{h}})\in V_{\mathbf{h}^{\sigma}}$,
we get that the functions $T_{\sigma^{-1},\mathbf{h}^{\sigma}}^{\prime}\circ\varphi_{\mathbf{h}^{\sigma}}|_{V_{\mathbf{h}}}:V_{\mathbf{h}}\to V_{\mathbf{h}}$
are in $C\left(\left\langle F\overline{G^{\prime}}\right\rangle _{\mathbf{h}}\right)$.
Therefore, by \propref{preparations_list_of_invariants}, 
\[
T_{\sigma^{-1},\mathbf{h}^{\sigma}}^{\prime}\circ\varphi_{\mathbf{h}^{\sigma}}|_{V_{\mathbf{h}}}=\sum_{\tau\in S_{\mathbf{h}}}\alpha_{\tau}T_{\mbox{\ensuremath{\tau},}\mathbf{h}}^{\prime}|_{V_{\mathbf{h}}}.
\]
By multiplying the above equality from the left by $T_{\sigma,\mathbf{h}}^{\prime}$,
we deduce: 
\[
\varphi_{\mathbf{h}^{\sigma}}|_{V_{\mathbf{h}}}=\sum_{\tau\in S_{\mathbf{h}}}\alpha_{\tau}T_{\sigma\mbox{\ensuremath{\tau},}\mathbf{h}}^{\prime}|_{V_{\mathbf{h}}}.
\]
We saw that $\varphi_{\mathbf{h}^{\sigma}}$ is zero when applied
to vectors \textbf{not }from $V_{\mathbf{h}}$. Hence,
\[
\varphi_{\mathbf{h}^{\sigma}}=\sum_{\tau\in S_{\mathbf{h}}}\alpha_{\tau}T_{\sigma\mbox{\ensuremath{\tau},}\mathbf{h}}^{\prime}=\sum_{\sigma\in S_{n}/\mathbf{h}}a_{\sigma,\mathbf{h}}v_{\mathbf{h}^{\sigma}}.
\]

Now consider \foreignlanguage{american}{the equality $TX_{g}=X_{g}T$,
where }$X_{g}=\sum_{t\in B}X_{gt,t}\in\overline{G}$, $g\in H_{\mathfrak{g}}$
and $\widehat{X}_{gt,t}$ is the identity matrix\foreignlanguage{american}{
for every $t\in B$.} 
\[
TX_{g}\left(v_{\mathbf{h}}\right)=T\left(X_{gh_{1},h_{1}}v_{h_{1}}\otimes\cdots\otimes X_{gh_{n}h_{n}}v_{h_{n}}\right)
\]
\[
=\sum_{\sigma\in S_{n}/\mathbf{h}}a_{\sigma,g\mathbf{h}}\left(X_{gh_{\sigma(1)},h_{\sigma(1)}}v_{h_{\sigma(1)}}\otimes\cdots\otimes X_{gh_{\sigma(n)}h_{\sigma(n)}}v_{h_{\sigma(n)}}\right)
\]
and \linebreak{}
\[
X_{g}T\left(v_{\mathbf{h}}\right)=X_{g}\left(\sum_{\sigma\in S_{n}/\mathbf{h}}a_{\sigma,\mathbf{h}}v_{\mathbf{h}^{\sigma}}\right)
\]
\[
=\sum_{\sigma\in S_{n}/\mathbf{h}}a_{\sigma,\mathbf{h}}\left(X_{gh_{\sigma(1)},h_{\sigma(1)}}v_{h_{\sigma(1)}}\otimes\cdots\otimes X_{gh_{\sigma(n)}h_{\sigma(n)}}v_{h_{\sigma(n)}}\right)
\]
forcing $a_{\sigma,g\mathbf{h}}=a_{\mathbf{\sigma,h}}$ for every
$g\in H_{\mathfrak{g}}$. 

All in all 
\[
T=\sum_{\sigma,\mathbf{h}}a_{\sigma,\mathbf{h}}T_{\sigma,\mathbf{h}}.
\]
\end{proof}
\selectlanguage{english}%
\begin{rem}
See also Regev's paper \cite{key-1} for similar ideas and techniques. 
\end{rem}

\subsection{Approximation of $t_{n}^{G}$}

The next step is to approximate $t_{n}^{G}(A)$ using the explicit
description of $T_{n}^{\widetilde{G}}(A)$ given in \thmref{centralizer}.
This will be carried out by appending to the ungraded case. More precisely,
we interpret $t_{n}^{G}(A)$ by an expression concerning terms of
the form $t_{l}(r)=t_{l}^{\{e\}}\left(M_{r}(F)\right)$ (i.e. the
dimension of the invariant spaces of ungraded $r\times r$ matrices),
thus essentially reducing to the ungraded case. Due to the extensive
study of the sequence $t_{l}(r)$ by \foreignlanguage{american}{Procesi
\cite{key-Procesi}, Razmyslov \cite{key-Raz}}, and Regev \cite{key-Regev2},
we will be able to use this reduction to conclude an asymptotic expression
for $t_{n}^{G}(A)$. 
\begin{rem}
For purposes of simplification, we set $t_{0}(r)=1$. 
\end{rem}
First we consider a slight modification of the spaces $T_{n}^{\widetilde{G}}$. 
\selectlanguage{american}%
\begin{prop}
\label{prop:Calc} Let $I{}_{n}^{\prime}=\sp\left\{ T{}_{\sigma,\mathbf{h}}^{\prime}|\sigma\in S_{n},\mathbf{h}\in B^{n}\right\} $.
Then, 
\[
\dim_{F}\left(I{}_{n}^{\prime}\right)=\sum_{n_{1}+\cdots+n_{k}=n}{n \choose n_{1},...,n_{k}}^{2}t_{n_{1}}\left(m_{1}\right)\cdots t_{n_{k}}\left(m_{k}\right)
\]
\end{prop}
\begin{proof}
Recall that $B=\{\gamma_{1},...,\gamma_{k}\}$ and fix a vector $\mathbf{g}\in B^{n}$
with exactly $n_{i}$ appearances of $\gamma_{i}$ where $i=1,...,k$.
Without loss of generality assume $\mathbf{g=}(\gamma_{1}^{n_{1}},...,\gamma_{k}^{n_{k}})$,
and let $\nu S$ be a left $S=S_{\{1,..,n_{1}\}}\times\cdots\times S_{\{n_{1}+\cdots+n_{k}+1,...,n\}}$
coset of $S_{n}$. We want to show that: 
\[
\dim_{F}\left(\sp\{T{}_{\sigma,\mathbf{g}}^{\prime}|\sigma\in\nu S\}\right)=t_{n_{1}}\left(m_{1}\right)\cdots t_{n_{k}}\left(m_{k}\right).
\]

Let $L_{s}(t)$ be the subspace of $\End\left(U_{t}^{\otimes s}\right)$
($U_{t}$ is a linear space of dimension $t$) spanned by the linear
maps $T_{\tau}$ ($\tau\in S_{s})$ which send $u_{1}\otimes\cdots\otimes u_{s}$
to $u_{\tau(1)}\otimes\cdots\otimes u_{\tau(s)}$ for every $u_{1},...,u_{s}\in U_{t}$.
Consider the linear mapping: 
\[
\phi:\sp\{T{}_{\sigma,\mathbf{g}}^{\prime}|\sigma\in\nu S\}\longrightarrow L_{n_{1}}(m_{1})\otimes\cdots\otimes L_{n_{k}}(m_{k})
\]
 defined by 
\[
\phi\left(T{}_{\nu\tau,\mathbf{g}}^{\prime}\right)=T_{\tau_{1}}\otimes\cdots\otimes T_{\tau_{k}},
\]
 where $\tau_{i}$ is the restriction of $\tau$ to $S_{\{n_{1}+\cdots+n_{i-1}+1,...,n_{1}+\cdots+n_{i}\}}$
(or $S_{\{1,..,n_{1}\}}$ for $i=1$). Since this mapping is evidently
an isomorphism of linear spaces, and since $\dim_{F}L_{n_{i}}(m_{i})=t_{n_{i}}\left(m_{i}\right)$
(e.g. see \cite{key-Procesi}), we indeed obtain: 
\[
\dim_{F}\left(\sp\{T{}_{\sigma,\mathbf{g}}^{\prime}|\sigma\in\nu S\}\right)=t_{n_{1}}\left(m_{1}\right)\cdots t_{n_{k}}\left(m_{k}\right).
\]

Now, there are ${n \choose n_{1},...,n_{k}}$ left $S$ cosets in
$S_{n}$, and obviously all the spaces $\sp\{T{}_{\sigma,\mathbf{g}}^{\prime}|\sigma\in\nu S\}$
corresponding to different cosets are linearly independent. So, for
any $\mathbf{g}$ with $n_{i}$ appearances of $g_{i}$ ($i=1...k$)
we have 
\[
\dim_{F}\left(\sp\{T{}_{\sigma,\mathbf{g}}^{\prime}\}\right)={n \choose n_{1},...,n_{k}}t_{n_{1}}\left(m_{1}\right)\cdots t_{n_{k}}\left(m_{k}\right).
\]
There are ${n \choose n_{1},...,n_{k}}$ such $\mathbf{g}$'s, and
all the spaces $\sp\{T{}_{\sigma,\mathbf{g}}^{\prime}\}$ corresponding
to different $\mathbf{g}$'s are linearly independent. So, for a fixed
$k$-tuple $(n_{1},...,n_{k})$: \foreignlanguage{english}{
\begin{equation}
\dim_{F}\left(\sp\left\{ T{}_{\sigma,\mathbf{g}}^{\prime}|\mbox{ \ensuremath{\mathbf{g}} has \ensuremath{n_{i}} appearances of \ensuremath{\gamma_{i}}}\right\} \right)={n \choose n_{1},...,n_{k}}^{2}t_{n_{1}}\left(m_{1}\right)\cdots t_{n_{k}}\left(m_{k}\right)\label{eq:ofir}
\end{equation}
Summing up on all the} possible $k$-tuples gives us: 
\[
\sum_{n_{1}+\cdots+n_{k}=n}{n \choose n_{1},...,n_{k}}^{2}t_{n_{1}}\left(m_{1}\right)\cdots t_{n_{k}}\left(m_{k}\right).
\]
\end{proof}
\begin{cor}
\label{cor:We-have-where}We have 
\[
t_{n}^{G}(A)=\frac{1}{|H_{\mathfrak{g}}|}\dim_{F}\left(I{}_{n}^{\prime}\right)=\frac{1}{|H_{\mathfrak{g}}|}\sum_{n_{1}+\cdots+n_{k}=n}{n \choose n_{1},...,n_{k}}^{2}t_{n_{1}}\left(m_{1}\right)\cdots t_{n_{k}}\left(m_{k}\right).
\]
\end{cor}
\begin{rem}
\label{rem:For-future-use}For future use note that by using \ref{eq:ofir},
one obtains
\[
\dim_{F}\left(\sp\left\{ T{}_{\sigma,\mathbf{g}}|\mbox{ \ensuremath{\mathbf{g}} has \ensuremath{n_{i}}appearances of \ensuremath{\gamma_{i}}}\right\} \right)={n \choose n_{1},...,n_{k}}^{2}t_{n_{1}}\left(m_{1}\right)\cdots t_{n_{k}}\left(m_{k}\right).
\]

\end{rem}
To calculate the asymptotics of the expression above we need the following
result of Regev and Beckner (\cite{Beckner=000026Regev} Theorem 1.2).
Here is a simplified version:
\begin{thm}
\label{thm:B-R} Let $p=(p_{1},...,p_{k})\in\mathbb{Q}^{k}$ such
that $\sum p_{i}=1$, and suppose $F(x_{1},...,x_{k})$ is a continuous
homogeneous function of degree $d$ with $0<F(p)<\infty.$ Then for
$\rho=d-\frac{1}{2}(\beta-1)(k-1)$ and $\beta>0$ 
\begin{eqnarray*}
 & \sum_{\begin{array}{c}
n_{1}+\cdots+n_{k}=n\\
n_{i}\ne0
\end{array}}\left[{n \choose n_{1},...,n_{k}}p_{1}^{n_{1}}\cdots p_{k}^{n_{k}}\right]^{\beta}F(n_{1},...,n_{k})\\
 & \sim n^{\rho}\beta^{-\frac{k-1}{2}}\left(\frac{1}{\sqrt{2\pi}}\right)^{(\beta-1)(k-1)}F(p)\left(\prod_{j=1}^{k}p_{j}\right)^{\frac{1-\beta}{2}}
\end{eqnarray*}
\end{thm}
\selectlanguage{english}%
\begin{cor}
\label{cor:codim-asymp}Let $G$ be a group, and $F$ a field of characteristic
zero. Let $A$ be the $F$-algebra of $m\times m$ matrices with elementary
$G$-grading, and let $\mathfrak{g}=(\gamma_{1}^{m_{1}},...,\gamma_{k}^{m_{k}})$
be the grading vector where all the $g_{i}$'s are distinct. Then
\[
t_{n}^{G}(A)=\alpha n^{-\frac{\sum_{i=1}^{k}m_{i}^{2}-1}{2}}m^{2n}=\alpha n^{\frac{1-dim_{F}A_{e}}{2}}(\dim_{F}A)^{n}
\]

where
\[
\alpha=\frac{1}{|H_{\mathfrak{g}}|}m^{\frac{\sum_{i=1}^{k}m_{i}^{2}}{2}}\left(\frac{1}{\sqrt{2\pi}}\right)^{m-1}\left(\frac{1}{2}\right)^{\frac{\sum m_{i}^{2}-1}{2}}\prod_{i=1}^{k}\left(1!2!\cdots(m_{i}-1)!m_{i}^{-\frac{1}{2}}\right).
\]
\end{cor}
\begin{proof}
First note that by \cite{key-Regev2} $t_{l}\left(s\right)\sim\beta_{s}l^{-\frac{s^{2}-1}{2}}s^{2l}$
where 
\[
\beta_{s}=\left(\frac{1}{\sqrt{2\pi}}\right)^{s-1}\left(\frac{1}{2}\right)^{\frac{s^{2}-1}{2}}1!2!\cdots(s-1)!s^{\frac{s^{2}}{2}}.
\]
Using standard calculus arguments, we deduce from \propref{Calc}
that
\[
t_{n}^{G}(A)\sim\frac{1}{|H_{\mathfrak{g}}|}\sum_{\begin{array}{c}
n_{1}+\cdots+n_{k}=n\\
n_{i}\ne0
\end{array}}\left({n \choose n_{1},...,n_{k}}^{2}\prod_{i=1}^{k}\beta_{m_{i}}n_{i}^{-\frac{m_{i}^{2}-1}{2}}m_{i}^{2n_{i}}\right).
\]
In other words $t_{n}^{G}(A)$ is asymptotically equal to 
\[
\frac{1}{|H_{\mathfrak{g}}|}\left(\prod_{i=1}^{k}\beta_{m_{i}}m^{2n}\right)\cdot\sum_{\begin{array}{c}
n_{1}+\cdots+n_{k}=n\\
n_{i}\ne0
\end{array}}\left(\left[{n \choose n_{1},...,n_{k}}\left(\frac{m_{1}}{m}\right)^{n_{1}}\cdots\left(\frac{m_{k}}{m}\right)^{n_{k}}\right]^{2}\prod_{i=1}^{k}n_{i}^{\frac{1-m_{i}^{2}}{2}}\right).
\]
Applying \thmref{B-R} on the last expression with $p=\left(\frac{m_{1}}{m},...,\frac{m_{k}}{m}\right)$
and $F(x_{1},...,x_{k})=\prod_{i=1}^{k}x_{i}^{\frac{1-m_{i}^{2}}{2}}$
(which is a \foreignlanguage{american}{continuous homogeneous function
of degree }$d=\frac{k-\sum_{i=1}^{k}m_{i}^{2}}{2}$) gives 
\[
t_{n}^{G}(A)\sim\frac{1}{|H_{\mathfrak{g}}|}\prod_{i=1}^{k}\beta_{m_{i}}\beta n^{\frac{1-\sum_{i=1}^{k}m_{i}^{2}}{2}}m^{2n}
\]
where (by \thmref{B-R})
\[
\gamma=2^{\frac{1-k}{2}}\left(\frac{1}{\sqrt{2\pi}}\right)^{(k-1)}F\left(\frac{m_{1}}{m},...,\frac{m_{k}}{m}\right)\left(\prod_{j=1}^{k}\frac{m_{j}}{m}\right)^{-\frac{1}{2}}
\]
A simple calculation shows $\alpha=\frac{1}{|H_{\mathfrak{g}}|}\beta\prod_{i=1}^{k}\beta_{m_{i}}$. 
\end{proof}

\section{\label{sec:-intro-rep}$GL_{m}(F)$ representations }

In order to show that the sequence $t_{n}^{G}(A)$ approximates the
codimension sequence, we need to use $GL_{m}(F)$-representation theory.
In this section we will recall its basic constructions and results.

\subsection{Rational and polynomial $GL_{m}(F)$-representations}

A vector space $Y=Sp_{F}\{y_{1},y_{2},...\}$ is a \emph{polynomial}
\textit{\emph{(respectively,}}\textit{ rational}\textit{\emph{)}}
$GL_{m}(F)$-module if for every $j>0$ and $P\in GL_{m}(F)$, the
$GL_{m}(F)$-action is given by $P\cdot y_{j}=\sum f_{i,j}(P)y_{i}$
(a finite sum), where $f_{i,j}(P)$ are polynomials (rational functions)
in the entries of the matrix $P$. We say that \emph{$Y$ is $r$-homogeneous}
if all the $f_{i,j}$'s are homogeneous of total degree $r$. It is
possible to decompose a $GL_{m}(F)$-module $Y$ to a direct sum of
its homogenous components: 
\[
Y=\bigoplus_{r\geq0}Y^{(r)}.
\]
Here $Y^{(r)}$ is the $GL_{m}(F)$-submodule consisting of all $y\in Y$
such that $P\cdot y=\sum f_{i,y}(P)y_{i}$, where $f_{i,y}(P)$ is
a homogenous function of degree $r$.

For $\alpha=(\alpha_{1},...,\alpha_{m})$, where the $\alpha_{i}$
are (non negative) integers, one defines \emph{the weight space of
$Y$ associated to $\alpha$ }by $Y^{\alpha}=\left\{ y\in Y|\textrm{diag}(p_{1},...,p_{m})\cdot y=p_{1}^{\alpha_{1}}\cdots p_{m}^{\alpha_{m}}y\right\} $.
It is known that the weight spaces satisfy: 
\[
Y=\bigoplus_{\alpha}Y^{\alpha}.
\]
There is a simple connection between the weight spaces and the homogenous
components: 
\[
Y^{(r)}=\bigoplus_{\alpha_{1}+\cdots\alpha_{m}=r}Y^{\alpha=(\alpha_{1},...,\alpha_{m})}.
\]

Suppose that $Y$ is a finite dimensional $GL_{m}(F)$-module and
$\chi_{Y}:GL_{m}(F)\longrightarrow F^{*}$ is the character of $Y$.
Then, from the representation theory of $GL_{n}(F)$, it is known
that the polynomial (rational function) 
\[
H_{Y}(t_{1},...,t_{m})=\sum_{\alpha}\left(\dim_{F}Y^{\alpha}\right)t_{1}^{\alpha_{1}}\cdots t_{m}^{\alpha_{m}}\in\mathbb{Z}\left[t_{1},...,t_{m}\right]\,\,\left(\mathbb{Z}\left[t_{1}^{\pm1},...,t_{m}^{\pm1}\right])\right)
\]
 is symmetric and 
\[
\chi_{Y}(P)=H_{Y}(p_{1},...,p_{m}),
\]
where $p_{1},...,p_{m}$ are the eigenvalues of $P\in GL_{m}(F)$.
This polynomial (rational function) is called the \textit{Hilbert
}\textit{\emph{(or}}\textit{ Poincare}\textit{\emph{)}}\textit{ series}
of $Y$. Therefore, if we denote by $\mathfrak{R}(GL_{m}(F))$ the
Grothendieck ring of finite dimensional polynomial (rational) $GL_{m}(F)$-modules,
we get a ring homomorphism $\Psi$ from $\mathfrak{R}(GL_{m}(F))$
to $\mathbb{Z}\left[t_{1},...,t_{m}\right]^{S_{n}}$ ($\mathbb{Z}\left[t_{1}^{\pm1},...,t_{m}^{\pm1}\right]^{S_{n}}$)
which sends $[Y]\in\mathfrak{R}(GL_{m}(F))$ to $H_{Y}$. This is,
in fact, an isomorphism. 

It is easy to extend the definition of a Hilbert function to the case
where $Y$ is not necessary finite dimensional but its homogenous
components are. Just consider the formal sum
\[
H_{Y}=\sum_{r\geq0}H_{Y^{(r)}}.
\]

\subsection{Schur functions}

The Hilbert series is known to be symmetric in $t_{1},...,t_{m}$,
so we recall an important $F$-linear \uline{basis} of the space
of symmetric polynomials,  the \emph{Schur functions}. For convenience
sake, we use the following combinatorial definitions.

A \textit{partition }is a finite sequence of integers $\lambda=(\lambda_{1},...,\lambda_{k})$
such that $\lambda_{1}\geq\cdots\geq\lambda_{k}>0$. The \emph{height}
\emph{of }$\lambda$ denoted by $\mbox{ht}(\lambda)$ is the integer
$k$. The set of all partitions is denoted by $\Lambda$, and the
set of all partitions of height less or equal to $k$ is denoted by
$\Lambda^{k}$.

We say that $\lambda$ is a\textit{ partition of $n\in\mathbb{N}$}
if $\sum_{i=1}^{k}\lambda_{i}=n$. In this case we write $\lambda\vdash n$
or $|\lambda|=n$. 

Let $\lambda=(\lambda_{1},..,\lambda_{k})$ be a partition. The \textit{Young
diagram} associated to $\lambda$ is the finite subset of $\mathbb{Z}\times\mathbb{Z}$
defined as $D_{\lambda}=\left\{ (i,j)\in\mathbb{Z}\times\mathbb{Z}|i=1,..,k\,,\, j=1,..,\lambda_{i}\right\} $.
We may regard $D_{\lambda}$ as $k$ arrays of boxes where the top
one is of length $\lambda_{1}$, the second of length $\lambda_{2}$,
etc. For example 

\begin{center}
$D_{(4,3,3,1)}=$ %
\begin{tabular}{|c|c|c|c}
\hline 
~ & ~ & ~ & \multicolumn{1}{c|}{~~}\tabularnewline
\hline 
 &  &  & \tabularnewline
\cline{1-3} 
 &  &  & \tabularnewline
\cline{1-3} 
 & \multicolumn{1}{c}{} & \multicolumn{1}{c}{} & \tabularnewline
\cline{1-1} 
\end{tabular}
\par\end{center}

\textit{A Schur function} $s_{\lambda}(t_{1},...,t_{m})$ is a symmetric
polynomial such that the coefficient of $t_{1}^{a_{1}}\cdots t_{m}^{a_{m}}$
is equal to the number of ways to insert $a_{1}$ ones, $a_{2}$ twos,
... , and $a_{m}$ $m$'s in $D_{\lambda}$ such that in every row
the numbers are non-decreasing, and in any column the numbers are
strictly increasing. Note that although this definition is not the
classical one, it is equivalent to it (e.g. see \cite{key-Bruce}). 
\begin{example*}

\begin{enumerate}
\item If $\lambda$ is a partition of height one, i.e $\lambda=(\lambda_{1})$,
then the corresponding Schur function is 
\[
s_{(\lambda_{1})}(t_{1},...,t_{m})=\sum_{a_{1}+\cdots+a_{m}=\lambda_{1}}t_{1}^{a_{1}}\cdots t_{m}^{a_{m}}.
\]

\item If $\lambda=(2,1)$ and $m=2$ we have: 
\[
s_{(2,1)}(t_{1},t_{2})=t_{1}^{2}t_{2}+t_{1}t_{2}^{2},
\]
since the only two ways to set ones and twos in $D_{(2,1)}$ are %
\begin{tabular}{|c|c|}
\hline 
1 & 1\tabularnewline
\hline 
2 & \multicolumn{1}{c}{}\tabularnewline
\cline{1-1} 
\end{tabular} and %
\begin{tabular}{|c|c|}
\hline 
1 & 2\tabularnewline
\hline 
2 & \multicolumn{1}{c}{}\tabularnewline
\cline{1-1} 
\end{tabular}. 
\end{enumerate}
\end{example*}
It turns out that the Schur functions correspond under $\Phi$ to
the irreducible $GL_{m}(F)$-representations. Therefore, if we define
a scalar product $<,>$ on $\mathbb{Z}[t_{1},...,t_{m}]^{S_{m}}$
by choosing $\{s_{\lambda}|\,\lambda\in\Lambda^{m}\}$ to be an orthonormal
basis, we make $\Phi$ into an \textbf{isometry}.

\subsection{Generalization to representations of products of general linear groups}

Similarly we can consider $GL_{m_{1}}(F)\times\cdots\times GL_{m_{k}}(F)$,
where $m=m_{1}+\cdots+m_{k}$, and obtain an isomorphic isometry:
\[
\mathfrak{\mathbf{\Phi}:R}(GL_{m_{1}}(F)\times\cdots\times GL_{m_{k}}(F))\to\mathbb{Z}[t_{1},...,t_{m}]^{S_{m_{1}}\times\cdots\times S_{m_{k}}},
\]
given by $[Y]\mapsto H_{Y}(t_{1},...,t_{m})=\sum_{\alpha}\left(\dim_{F}Y^{\alpha}\right)t_{1}^{\alpha_{1}}\cdots t_{m}^{\alpha_{m}}$,
where 
\[
Y^{\alpha}=\left\{ y\in Y|\textrm{diag}(p_{1},...,p_{m_{1}})\times\cdots\times\textrm{diag}(p_{m-m_{k}+1},...,p_{m})\cdot y=p_{1}^{\alpha_{1}}\cdots p_{m}^{\alpha_{m}}y\right\} 
\]
and the scalar product $<,>$ on the right ring is determined by choosing:
\[
\left\{ s_{\lambda_{1}}(t_{1},...,t_{m_{1}})\cdots s_{\lambda_{k}}(t_{m-m_{k}+1},...,t_{m})|\lambda_{1}\in\Lambda^{m_{1}},...,\lambda_{k}\in\Lambda^{m_{k}}\right\} ,
\]
to be an orthonormal basis.

\subsection{Connection to representations of $S_{m}$}

There is a connection between the representations of $S_{m}$ and
$GL_{m}(F)$. Without getting into much details, there is an isometry:
\[
\rho:\mathfrak{R}^{(m)}(GL_{m}(F))\longrightarrow\mathfrak{R}(S_{m}),
\]
where $\mathfrak{R}^{(r)}(GL_{m}(F))$ is the subring of $\mathfrak{R}(GL_{m}(F))$
consisting of $r$-homogeneous modules (equivalence classes) and $\mathfrak{R}(S_{m})$
is the Grothendieck ring of finite dimensional $S_{m}$-modules. The
isometry $\rho$ is given by $\rho(\left[Y\right])=\left[Y^{(1^{m})}\right]$
($(1^{m})$ means $(1,...,1)$ - $m$ times), where $S_{m}$ acts
on $Y^{(1^{m})}$ via permutation matrices (this is indeed an action
since permutation matrices and diagonal matrices commute). Therefore,
if $H_{Y}=\sum_{\lambda}a_{\lambda}s_{\lambda}$ is a Hilbert series
of some $GL_{m}(F)$-module $Y$, we obtain 
\[
\rho(Y^{(m)})=\sum_{\htt(\lambda)=m}a_{\lambda}\chi_{\lambda},
\]
where $\chi_{\lambda}$ is the character of the irreducible $S_{n}$-module
corresponding to $s_{\lambda}$.

Finally, for $r\geq m$ there is an epimomorphism: 
\[
\mathfrak{R}(GL_{r}(F))\longrightarrow\mathfrak{R}(GL_{m}(F))
\]
taking $\left[Y\right]$ to $\left[\oplus_{\alpha}Y^{\alpha}\right]$,
where $\alpha$ varies over all the weights of length $r$ of the
form $(\alpha_{1},...,\alpha_{n},0,...0)$. The action of $GL_{m}(F)$
on this space is given by the identification of $GL_{m}(F)$ with
the subgroup $\left\{ \left(\begin{array}{cc}
P & 0\\
0 & I
\end{array}\right)|P\in GL_{m}(F)\right\} $ of $GL_{r}(F)$. 
\begin{rem}
This homomorphism is \textbf{not} an isometry. However, if we restrict
ourselves to modules without appearances of irreducible modules corresponding
to $\lambda$ with $\htt(\lambda)>m$ in their decomposition (to direct
sum of irreducible modules), then we do get an isometry. 
\end{rem}
Translating these to Schur functions we conclude that $s_{\lambda}(t_{1},...,t_{m})=s_{\lambda}(t_{1},...,t_{m},\underset{r-m}{\underbrace{0,..,0}})$
for $\lambda$ having height $\leq m$. As a consequence we may sometimes
allow ourselves not to specify the indeterminates $t_{1},...,t_{r}$
and write simply $s_{\lambda}$.

\section{\label{sec:Hilbert-series}Hilbert series of the ring of invariants}

\emph{For the sake of convenience, from now on, we denote the invariant
spaces $I_{n}$ instead of $T_{n}^{\widetilde{G}}$. }

Showing directly that $t_{n}^{G}(A)(=\dim_{F}I_{n})$ approximates
$c_{n}^{G}(A)$ seems to be a very hard. To overcome this difficulty
we \textit{\emph{consider the sequence of }}\textit{special invariants}
\[
SI_{n}=\sp\left\{ T_{\sigma,\mathbf{g}}\in I_{n}|\mbox{ \ensuremath{\sigma\in S_{n}\mbox{ is an \ensuremath{n}-cycle}}}\right\} 
\]
 and show that 
\[
\dim_{F}I_{n}\sim\dim_{F}SI_{n}\sim c_{n}^{G}(A).
\]
 In this section we use tools from representation theory of $GL_{m}(F)$
which were introduced in the previous section in order to deduce $\dim_{F}I_{n}\sim\dim_{F}SI_{n}$.
The rest of the proof is in the next section.

Suppose $V$ is a vector space over $F$. The polynomial ring over
$V$, denoted by $F[V]$ is defined to be $S(V^{*})$, the symmetric
tensor algebra over $V^{*}$. This can be made more concrete by choosing
coordinates, i.e. a basis $\{v_{1},v_{2},...\}$ of $V$ and interpreting
$F[V]$ as the polynomial ring over $F$ with commuting variables
corresponding to the $v_{i}$'s. Our main interest is in the case
where $V$ is the underlying vector space of $A^{\times n}=A\otimes F^{\times n}=M_{m}(F)\times\cdots\times M_{m}(F)$.
To connect this algebra to the main theme of this paper, notice that
the subspace of multilinear polynomials in $F[A^{\times n}]$ is exactly
the space $T_{n}=\left(A^{\otimes n}\right)^{*}$ (see \secref{Motivation})
whose space of invariance under the action of $\widetilde{G}$ is
$I_{n}$. The strategy is to consider the action of $\widetilde{G}$
on $F[A^{\times n}]$ together with a new action of the group $GL_{n}(F)$,
which cannot be defined on $\left(A^{\otimes n}\right)^{*}$ but on
$F[A^{\times n}]$, in order to get a better insight on $I_{n}$ and
$SI_{n}$. To be more precise, we define an action of $\widetilde{G}\times GL_{n}(F)$
on $A\otimes F^{\times n}$ by $(Q,P)\cdot a\otimes v=Q(a)\otimes Pv$
($Q\in\overline{G},P\in GL_{n}(F)$), obtaining an induced action
of $\widetilde{G}\times GL_{n}(F)$ on $F[A^{\times n}]$.

\subsection{Setting the stage}

Consider the following basis for $A^{\times n}$: 
\[
\left\{ (0,...,0,\underset{l\mbox{'th place}}{\underbrace{e_{s(i),t(j)}}},0,...,0)\,|\, s,t\in B\,1\leq i\leq m_{s},\,1\le j\leq m_{t}\,,\,1\leq l\leq n\right\} 
\]
(we use here notations from \secref{Invariants}), and denote by $u_{s(i),t(j)}^{(l)}$
the corresponding coordinates in $F[A^{\times n}]$. In this notation
$T_{n}$ is \textbf{identified} with the space of multilinear polynomials
in $F[A^{\times n}]$ i.e. 
\[
T_{n}\cong\sp\left\{ u_{s_{1}(i_{1}),t_{1}(j_{1})}^{(1)}\cdots u_{s_{n}(i_{n}),t_{n}(j_{n})}^{(n)}\:|\: s_{l},t_{l}\in B,\:1\le i_{l}\leq m_{s_{l}},\:1\leq j_{l}\leq m_{t_{l}}\right\} .
\]

Let $\mathfrak{L}_{N,n}=\{1,...,n\}^{\{1,..,N\}}$ (i.e. set theoretical
functions from $\{1,...,N\}$ to $\{1,...,n\}$). For every $\xi\in\mathfrak{L}_{N,n}$
and $u_{s(i),t(j)}^{(l)}\in A^{\times N}$ define $\xi(u_{s(i),t(j)}^{(l)})=u_{s(i),t(j)}^{(\xi(l))}$
and expand to a linear operator form $F[A^{\times N}]$ to $F[A^{\times n}]$.
Consider the $F$-spaces of (non-multilinear) invariants: 
\[
RI_{n}=\bigoplus_{N\geq n}\mathfrak{L}_{N,n}I_{N}=F[A^{\times n}]^{\widetilde{G}},
\]
where $\mathfrak{L}_{N,n}I_{N}=\left\{ \xi f|\, f\in I_{N}\mbox{ and \ensuremath{\xi\in\mathfrak{L}_{N,n}}}\right\} $
and the space of (non-multilinear) spacial invariants
\[
RSI_{n}=\bigoplus_{N\geq n}\mathfrak{L}_{N,n}SI_{N}\subset RI_{n}.
\]
These spaces are easily seen to be $GL_{n}(F)$-polynomial modules.
Moreover, the weight space $(RI_{n})^{\alpha}$ ($\alpha=(\alpha_{1},...,\alpha_{n})\in\mathbb{Z}_{+}^{n}$)
consists of all the polynomials in $RI_{n}$ having $\alpha_{l}$
appearances of elements of the form $u_{s(i),t(j)}^{(l)}$ in every
monomial, for all $l=1...n$. Therefore, $\rho\left([RI_{n}^{n}]\right)=\left[\left(RI_{n}^{n}\right)^{(1^{n})}\right]=[I_{n}]$,
where $RI_{n}^{n}$ is the $n$'th homogenous component of $RI_{n}$,
and similarly $\rho\left([RSI_{n}^{n}]\right)=[SI_{n}]$. The corresponding
$S_{n}$ action on $I_{n}$ and $SI_{n}$ is given by 
\[
\tau\left(u_{s_{1}(i_{1}),t_{1}(j_{1})}^{(1)}\cdots u_{s_{n}(i_{n}),t_{n}(j_{n})}^{(n)}\right)=u_{s_{1}(i_{1}),t_{1}(j_{1})}^{(\tau^{-1}(1))}\cdots u_{s_{n}(i_{n}),t_{n}(j_{n})}^{(\tau^{-1}(n))}.
\]

\begin{defn}
Denote: 
\[
T_{r,n}=\sp\left\{ u_{s_{r+1}(i_{r+1}),t_{r+1}(j_{r+1})}^{(r+1)}\cdots u_{s_{n}(i_{n}),t_{n}(j_{n})}^{(n)}\:|\: s_{l},t_{l}\in B,\:1\le i_{l}\leq m_{s_{l}},\:1\leq j_{l}\leq m_{t_{l}}\right\} 
\]
and
\[
I_{r,n}=T_{r,n}^{\widetilde{G}}\cong\left\{ T_{\sigma,\mathbf{g}}|\mathbf{g}=(g_{r+1},...,g_{n}),\sigma\in S_{\{r+1,...,n\}}\right\} 
\]

\end{defn}

\subsection{\label{sub:The-isomorphism-between}The isomorphism between $T_{n}$
and $\protect\End_{F}(V^{\otimes n})$}

Let $V$ be an $m$ dimensional $F$-vector space graded by $G$,
as in \secref{Invariants}. Denote by $V_{1},...,V_{n}$ - $n$ copies
of $V$ and consider the map: 
\[
\pi:T_{n}\longrightarrow\End_{F}(V_{1}\otimes\cdots\otimes V_{n})
\]
given by

\[
u_{s_{1}(i_{1}),t_{1}(j_{1})}^{(1)}\cdots u_{s_{n}(i_{n}),t_{n}(j_{n})}^{(n)}\longmapsto T'_{\mathbf{s(}\mathbf{i),}\mathbf{t}\mathbf{(j)}}
\]
where, 
\[
T'_{\mathbf{s(i),t(j)}}(e_{h_{1}(k_{1})}\otimes\cdots\otimes e_{h_{n}(k_{n})})=\delta_{\mathbf{s(i),h(k)}}e_{\mathbf{t(j)}}
\]
 for every $\mathbf{h}=(h_{1},...,h_{n})\in B^{n}$ and $\mathbf{k}=(k_{1},...,k_{n})\in M_{\mathbf{h}}=\{1,...,m_{h_{1}}\}\times\cdots\times\{1,...,m_{h_{n}}\}$. 

It is not difficult (but tedious) to see that this is indeed an isomorphism
of $GL_{m}(F)$-modules. Similarly, we have an isomorphism $\pi:T_{r,n}\to\End_{F}(V_{r+1}\otimes\cdots\otimes V_{n})$. 

Finally, by viewing the spaces $\End_{F}(V_{1}\otimes\cdots\otimes V_{r})$
and $\End_{F}(V_{r+1}\otimes\cdots\otimes V_{n})$ as subspaces of
$\End_{F}(V^{\otimes n})=\End_{F}(V_{1}\otimes\cdots\otimes V_{n})$
in the obvious way, we obtain:
\[
\pi(f_{1})\circ\pi(f_{2})=\pi(f_{1}f_{2}),
\]
where $f_{1}\in T_{r}$ $f_{2}\in T_{r,n}$.
\begin{lem}
\label{lem:mult-in-RI_n} $T_{\sigma,\mathbf{g}}T_{\tau,\mathbf{h}}=\sum_{h\in H_{\mathfrak{g}}}T_{\sigma\tau,\mathbf{g+}h\mathbf{h}},$
where $\sigma\in S_{r},\mathbf{g}=(g_{1},...,g_{r}),\tau\in S_{\{r+1,...,n\}}$
and $\mathbf{h}=(h_{r+1},...,h_{n})$. Here, $\mathbf{g}+\mathbf{h}$
is the concatenation of $\mathbf{g}$ with $\mathbf{h}$.\end{lem}
\begin{proof}
For all $g,t\in H_{\mathfrak{g}}$, and $v_{s}\in V_{s}$ ($s\in B$):
\[
T_{\sigma,\mathbf{g}}T_{\tau,\mathbf{h}}\left((v_{gg_{1}}\otimes\cdots\otimes v_{gg_{r}})\otimes(v_{th_{r+1}}\otimes\cdots\otimes v_{th_{n}})\right)=
\]
\[
(v_{gg_{\sigma(1)}}\otimes\cdots\otimes v_{gg_{\sigma(r)}})\otimes(v_{th_{\tau(r+1)}}\otimes\cdots\otimes v_{th_{\tau(n)}})
\]
\[
=T_{\sigma\tau,\mathbf{g+}g^{-1}t\mathbf{h}}\left(v_{gg_{1}}\otimes\cdots\otimes v_{gg_{r}}\otimes v_{g\left(g^{-1}th_{r+1}\right)}\otimes\cdots\otimes v_{g\left(g^{-1}th_{n}\right)}\right).
\]
Whereas 
\[
0=T_{\sigma,\mathbf{g}}T_{\tau,\mathbf{h}}\left((v_{g'_{1}}\otimes\cdots\otimes v_{g'_{r}})\otimes(v_{h'_{r+1}}\otimes\cdots\otimes v_{h'_{n}})\right)
\]
\[
=\left(\sum_{h\in H_{\mathfrak{g}}}T_{\sigma\tau,\mathbf{g+}h\mathbf{h}}\right)\left(v_{g'_{1}}\otimes\cdots\otimes v_{g'_{r}}\otimes v_{h'_{r+1}}\otimes\cdots\otimes v_{h'_{n}}\right).
\]
given $(g'_{1},...,g'_{r})\notin H_{\mathfrak{g}}\mathbf{g}$ or $(h'_{r+1},...,h'_{n})\notin H_{\mathfrak{g}}\mathbf{h}$. 
\end{proof}

\subsection{The Hilbert series $H_{RI_{n}}$}

The next step is to compute the Hilbert series of $RI_{n}$ when considered
as an $GL_{n}(F)$-module \textbf{in the case that $\overline{G}$
is connected} i.e. $\overline{G}\cong GL_{m_{1}}(F)\times\cdots\times GL_{m_{k}}(F)$,
so $H_{\mathbf{g}}=\{e\}$. The other cases will be discussed later.
Although the computation is similar to the one Formanek carried out
in (\cite{key-Formanek} pages 201,202), for the benefit of the reader
the details are presented bellow.

Consider $F[A^{\times n}]$ as an $GL_{m_{1}}(F)\times\cdots\times GL_{m_{k}}(F)\times GL_{n}(F)$-module.
Since $(P_{1},...,P_{k},Q)$ has eigenvalues $y_{i}y_{j}^{-1}t_{l}$
($i,j=1...m$ and $l=1...n$) when acting on $A^{\times n}$, where
$y_{1},...,y_{m},t_{1},...,t_{n}$ are eigenvalues of $(P_{1},...,P_{k},Q)\in GL_{m+n}(F)$,
the Hilbert series (with respect to $GL_{m_{1}}(F)\times\cdots\times GL_{m_{k}}(F)\times GL_{n}(F)$)
of $F[A^{\times n}]^{(r)}$, the space of polynomials of total degree
$r$ inside $F[A^{\times n}]$, is 
\[
\sum_{\alpha}\prod_{i,j,l}\left(y_{i}y_{j}^{-1}t_{l}\right)^{\alpha_{i,j,l}},
\]
where the sum is over all the vectors of non-negative integers $\alpha=(\alpha_{i,j,l})_{i,j=1...m,l=1...n}$
such that $\sum\alpha_{i,j,l}=r$. Therefore, 
\[
H_{F[A^{\times n}]}(y_{1},...,y_{m},t_{1},...,t_{n})=\prod_{i,j,l}\left(1+(y_{i}y_{j}^{-1}t_{l})^{1}+(y_{i}y_{j}^{-1}t_{l})^{2}+\cdots\right)
\]
which is, by \cite{key-MacBook}(p. 33), equals to
\[
\prod_{i,j,l}\frac{1}{1-y_{i}y_{j}^{-1}t_{l}}=\sum_{\lambda}s_{\lambda}(y_{i}y_{j}^{-1})s_{\lambda}(t_{1},...,t_{n}),
\]
where $s_{\lambda}(y_{i}y_{j}^{-1})=s_{\lambda}(y_{i}y_{j}^{-1}|i,j=1,...,m)$
is the Schur function with $m^{2}$ variables after a substitution
of the $y_{i}y_{j}^{-1}$'s. Hence, the Hilbert series of $F[A^{\times n}]^{\widetilde{G}}=RI_{n}$
is 
\[
H_{RI_{n}}(t_{1},...,t_{n})=\sum_{\lambda}\underset{a_{\lambda}}{\underbrace{\left\langle s_{\lambda}(y_{i}y_{j}^{-1}),1\right\rangle }}s_{\lambda}(t_{1},...,t_{n}).
\]

\begin{rem}
Notice that we do not take the $G$-grading into account in the Hilbert
series $H_{F[A^{\times n}]}$ and $H_{RI_{n}}$. \end{rem}
\begin{defn}
Let $\mu=(\mu_{1},...,\mu_{n})$ and $\lambda=(\lambda_{1},...,\lambda_{n^{\prime}})$
be two partitions and assume that $n^{\prime}\geq n$. The partition
$\lambda+\mu$ is defined to be 
\[
(\lambda_{1}+\mu_{1},...,\lambda_{n}+\mu_{n},\mu_{n+1},...,\mu_{n^{\prime}}).
\]
Moreover, for an integer $l\in\mathbb{N}$, we define $l\mu$ to be
the partition $(l\mu_{1},...,l\mu_{n})$. \end{defn}
\begin{lem}
\label{lem:mu-la}Let $\mu=(1^{m^{2}})$ and $l\in\mathbb{N}$.
\begin{enumerate}
\item If $\htt(\lambda)>m^{2}$ then $a_{\lambda}=0$.
\item $a_{\lambda}=a_{\lambda+l\mu}$ for every $\lambda$.
\end{enumerate}
\end{lem}
\begin{proof}
By definition (of a Schur function) if $\htt(\lambda)>r$ then $s_{\lambda}(t_{1},...,t_{r})=0$
for every $r\in\mathbb{N}$. In particular, if $\htt(\lambda)>m^{2}$
then $s_{\lambda}(y_{i}y_{j}^{-1})=0$ and so is $a_{\lambda}$.

For (2) observe that $s_{l\mu}(t_{1},...,t_{r})=(t_{1}\cdots t_{r})^{l}$,
and $s_{\lambda}s_{l\mu}=s_{\lambda+l\mu}$. Thus, $s_{l\mu}(y_{i}y_{j}^{-1})=\left(\prod_{i,j=1}^{m}y_{i}y_{j}^{-1}\right)^{l}=1$
and we obtain: 
\[
a_{\lambda}=\left\langle s_{\lambda}(y_{i}y_{j}^{-1}),1\right\rangle =\left\langle s_{\lambda}(y_{i}y_{j}^{-1})s_{l\mu}(y_{i}y_{j}^{-1}),1\right\rangle =\left\langle s_{\lambda+l\mu}(y_{i}y_{j}^{-1}),1\right\rangle =a_{\lambda+l\mu}.
\]

\end{proof}
Next, let $R_{n}=F\left\langle U^{(1)},...,U^{(n)}\right\rangle $
be the ring generated by $n$ generic $m\times m$ matrices $U^{(r)}=(u_{s(i),t(j)}^{(r)})$,
and let $\Cent{}_{n}$ be its center. Since an element $\beta\in\Cent{}_{n}$
is a scalar matrix, there is a polynomial $p\in F\left\langle x_{1},...,x_{n}\right\rangle $
($x_{i}$ are non-commutative variables) and $f\in F[A^{\times n}]$
such that $\beta=f\cdot I=p(U^{(1)},...,U^{(n)})$. By taking traces
of both sides we get the identity $f=\frac{1}{m}tr(p(U^{(1)},...,U^{(n)}))$.
This allows us to identify $\Cent{}_{n}$ (via $f\cdot I\mapsto f$)
with a subring of $F[A^{\times n}]$. In fact, for every group $H$
acting on $A$ as $F$-algebra automorphisms we may regard $\Cent_{n}$
as a subring of $F[A^{\times n}]^{H}$. In particular, this is true
in the case of $H=GL_{m}(F)$. 
\begin{rem}
\label{rem:C_n^ecycle}By \thmref{centralizer} it is possible to
represent the multilinear elements of $F[A^{\times n}]^{GL_{m}(F)}$
by linear combinations of $T_{\sigma}\in\End_{F}(V^{\otimes n})$,
where $T_{\sigma}(v_{1}\otimes\cdots\otimes v_{n})=v_{\sigma(1)}\otimes\cdots\otimes v_{\sigma(n)}$.
Notice, using \subref{The-isomorphism-between}, that $T_{\sigma}$
corresponds to 
\[
\sum_{1\leq i_{1},...,i_{n}\leq m}u_{i_{1},i_{\sigma(1)}}^{(1)}\cdots u_{i_{n},i_{\sigma(n)}}^{(n)}.
\]
If $\sigma$ decomposes to the disjoint product of cycles $(\sigma_{1}...\sigma_{a})\cdots(\sigma_{y}...\sigma_{z})$,
we can rewrite the last expression as 
\[
tr(U_{\sigma_{1}}\cdots U_{\sigma_{a}})\cdots tr(U_{\sigma_{y}}\cdots U_{\sigma_{z}}),
\]
 where the $U_{t}$ are generic matrices of $M_{m}(F)$ with entries
$u_{j_{1},j_{2}}^{(t)}$, where $1\leq j_{1},j_{2}\leq m$ and $t=1...m$.
Using this observation, Procesi and Razmyslov (\cite{key-Procesi,key-Raz})
showed that $F[A^{\times n}]^{GL_{m}(F)}$ is generated by traces
of generic matrices, and that the traces of multilinear monomials
correspond  to $T_{\sigma}$, where $\sigma\in S_{n}$ is a cycle.
From this it is easy to see that the multilinear elements of $\Cent{}_{n}\subset F[A^{\times n}]^{GL_{m}(F)}$
are given by linear combinations of $T_{\sigma}$ where $\sigma$
is an $n$-cycle. 
\end{rem}

\begin{rem}
Every element $T_{\sigma}\in\End_{F}(V^{\otimes n})$ is equal to
$\sum_{\mathbf{g}\in B^{n}}T_{\sigma,\mathbf{g}}$ ,where $T_{\sigma,\mathbf{g}}\in I_{n}$
.
\end{rem}
Formanek \cite{key-Formanek} showed that there is a (multilinear)
$\mathcal{I}\in F\left[A^{\times m^{2}}\right]^{GL_{m}(F)}$ such
that $F\mathcal{I}$ is one dimensional $GL_{m^{2}}(F)$-module i.e.
its corresponding Schur function is $s_{l\mu}$ for some $l\in\mathbb{N}$,
and for every $n\geq m^{2}$:
\begin{equation}
\mathcal{I}F[A^{\times n}]^{GL_{m}}\subseteq\Cent_{n}\subseteq F[A^{\times n}]^{GL_{m}}.\label{nek}
\end{equation}

\begin{rem}
In fact, $l=2$. Indeed, in the remark after the proof of theorem
16 in \cite{key-Formanek}, it is shown that $l=2$ if a certain polynomial
is not an identity of $M_{m}(F)$ - a fact whose proof can be found
in \cite{key-F2}.\end{rem}
\begin{thm}
The above $\mathcal{I}$ satisfies for every $n\geq m^{2}$: 
\[
\mathcal{I}\cdot RI_{n}\subseteq RSI_{n}\subseteq RI_{n}.
\]
\end{thm}
\begin{proof}
First, since $\overline{G}\subseteq GL_{m}(F)$ and $\mathcal{I}\in F\left[A^{\times m^{2}}\right]^{GL_{m}(F)}$
is multilinear, it is evident that $\mathcal{I}\in I_{m^{2}}$. Furthermore,
let $T_{\tau,\mathbf{h}_{0}}\in T_{m^{2},m^{2}+r}^{\widetilde{G}}$
and suppose $\mathcal{I}$ is given by $\sum_{i}k_{i}T_{\sigma_{i},\mathbf{g}_{i}}$.
From (\ref{nek}) we have that $\mathcal{I}T_{\tau}\in\Cent{}_{r+m^{2}}$.
Therefore, by \remref{C_n^ecycle} 
\[
\sum_{\mathbf{h}\in B^{r}}\mathcal{I}T_{\tau,\mathbf{h}}=\mathcal{I}T_{\tau}=\sum_{\nu\in S_{m^{2}+r}\mbox{ }}b_{\nu}T_{\nu}=\sum_{\nu\in S_{m^{2}+r}\mbox{ }}\sum_{\mathbf{t\in}B^{r+m^{2}}}b_{\nu}T_{\nu,\mathbf{t}}
\]
 where the $\nu$'s are $n$-cycles, and $b_{\nu}\in F$. 

Notice that by \lemref{mult-in-RI_n} we have 
\[
\mathcal{I}T_{\tau,\mathbf{h}}\in\sp\left\{ T_{\sigma,\mathbf{g}+\mathbf{h}}|\mathbf{g}\in B^{r},\sigma\in S_{r+m^{2}}\right\} .
\]
 Thus, $\mathcal{I}T_{\tau,\mathbf{h}}(x)\ne0$ only if 
\[
x\in V^{\otimes m^{2}}\otimes V_{h_{1}}\otimes\cdots\otimes V_{h_{r}}:=V_{\mathbf{h}}.
\]
Moreover, it is easy to see that $V^{\otimes(m^{2}+r)}=\oplus_{\mathbf{h\in}B^{r}}V_{\mathbf{h}}$.
Therefore, 
\[
\mathcal{I}T_{\tau,\mathbf{h}_{0}}=\sum_{\nu\in S_{m^{2}+r}}\sum_{\mathbf{t=g+}\mathbf{h_{0}}}b_{\nu}T_{\nu,\mathbf{t}}\in SI_{r+m^{2}}.
\]
Thus for every $f\in T_{m^{2},m^{2}+r}^{\widetilde{G}}$ we have $\mathcal{I}f\in SI_{r+m^{2}}$. 

Now, let $f\in RI_{n}$. Recall that by definition there is $\hat{\xi}\in\mathfrak{L}_{r,n}$
and $\hat{g}\in I_{r}$ such that $f=\hat{\xi(}\hat{g})$. By changing
the map $\hat{\xi}$ to be $\xi:\{m^{2}+1,...,m^{2}+r\}\to\{1,...n\}$;
$\xi(m^{2}+i)=\hat{\xi}(i)$, we may replace $\hat{g}$ by $g\in T_{m^{2},m^{2}+r}^{\widetilde{G}}$
such that $f=\xi(g)$. From the discussion above $\mathcal{I}g\in SI_{r+m^{2}}$,
thus $\mathcal{I}f=\xi'\left(\mathcal{I}g\right)$, where $\xi'\in\mathfrak{L}_{r+m^{2},n}$
takes $1,...,m^{2}$ to itself and $m^{2}+i$ to $\xi(m^{2}+i)$.
We conclude that $\mathcal{I}f\in RSI_{n}$ .
\end{proof}
Let $H_{\mathcal{I}RI_{n}}(t_{1},...,t_{n})=\sum_{\lambda}\bar{a}_{\lambda}s_{\lambda}$,
$H_{RI_{n}}(t_{1},...,t_{n})=\sum_{\lambda}a_{\lambda}s_{\lambda}$
and $H_{RSI_{n}}(t_{1},...,t_{n})=\sum_{\lambda}a{}_{\lambda}^{\prime}s_{\lambda}$.
The last theorem implies $\bar{a}_{\lambda}\leq a'_{\lambda}\leq a_{\lambda}$,
and by \lemref{mu-la} 
\[
\sum_{\lambda}a_{\lambda+2\mu}s_{\lambda+2\mu}=\sum_{\lambda}a_{\lambda}s_{\lambda+2\mu}=s_{2\mu}\sum_{\lambda}a_{\lambda}s_{\lambda}=H_{F\mathcal{I}\otimes RI_{n}}=H{}_{\mathcal{I}RI_{n}}=\sum_{\lambda}\bar{a}_{\lambda}s_{\lambda}.
\]

\begin{cor}
\textup{\label{cor:m=00003Dm'}For every partition $\lambda$ such
that $\lambda_{m^{2}}\geq l=2$ we have $a_{\lambda}=a{}_{\lambda}^{\prime}=\bar{a}_{\lambda}$.}\end{cor}
\begin{proof}
We see that $a_{\lambda}=\bar{a}_{\lambda}$ for every $\lambda$
which is of the form $\lambda=\lambda_{0}+2\mu$. Thus, this holds
for $\lambda$ such that $\lambda_{m^{2}}\geq2$. Since $a_{\lambda}\leq a_{\lambda}^{\prime}\leq\overline{a}_{\lambda}$,
we conclude that $a_{\lambda}=a{}_{\lambda}^{\prime}=\bar{a}_{\lambda}$
for every $\lambda$ such that $\lambda_{m^{2}}\geq2$.
\end{proof}
To use the Hilbert series to estimate $\dim_{F}I_{n}$ we need the
following theorem from $S_{n}$-representation theory:
\begin{thm}
\label{thm:S_n-char}The $S_{n}$-character of $I_{n}$ (resp. $SI_{n})$
is $\sum_{\lambda\vdash n}a{}_{\lambda}^{\prime}\chi_{\lambda}$ (resp.$\sum_{\lambda\vdash n}a{}_{\lambda}\chi_{\lambda}$).\end{thm}
\begin{proof}
It follows by means of the map $\rho$ in \secref{-intro-rep}. 
\end{proof}
In order to prove that $\dim_{F}I_{n}\sim\dim{}_{F}SI_{n}$ we want
to use a similar argument to the one given in \cite{key-10} (Theorem
5.10.3.). To do so the following lemma is key. 
\begin{lem}
\label{lem:m-poly}There is a polynomial $p(t)$ such that $\sum_{\lambda\vdash n}a_{\lambda},\sum_{\lambda\vdash n}a{}_{\lambda}^{\prime}<p(n)$
for every $\lambda\vdash n$.\end{lem}
\begin{proof}
Recall that from \lemref{mu-la} 
\[
H_{RI_{m^{2}}}(t_{1},...,t_{m^{2}})=\sum_{\lambda\in\Lambda^{m^{2}}}a_{\lambda}s_{\lambda}(t_{1},..,t_{m^{2}}).
\]

Therefore, if $M_{\lambda}$ is the irreducible $GL_{m^{2}}(F)$-representation
corresponding to $s_{\lambda}(t_{1},...,t_{m^{2}})$ then 
\[
\bigoplus_{\lambda\vdash n}a_{\lambda}M_{\lambda}=RI_{m^{2}}^{(n)}\subset F\left[A^{\times m^{2}}\right]^{(n)}.
\]
Thus, 
\[
\sum_{\lambda\vdash n}a{}_{\lambda}^{\prime}\leq\sum_{\lambda\vdash n}a_{\lambda}\leq\dim_{F}\left(F\left[A^{\times m^{2}}\right]^{(n)}\right)\leq{n+m^{4}-1 \choose n}=p(n).
\]

\end{proof}
The next definition will be very useful. 
\begin{defn}
Let $\{a_{n}\}_{n=1}^{\infty}$ be a series of real numbers. Set $\exp(a_{n}):=\limsup\sqrt[n]{a_{n}}$.
If $a_{n}=\dim_{F}(A_{n})$, where $\left\{ A_{n}\right\} _{n=1}^{\infty}$
is a series of vector spaces, then we set $\exp(A_{n}):=\exp(a_{n})$. \end{defn}
\begin{thm}
\textup{\label{thm:.glory}$\dim_{F}I_{n}\sim\dim_{F}SI_{n}$.}\end{thm}
\begin{proof}
Recall that from \thmref{S_n-char}
\[
\dim_{F}I_{n}=\sum_{\lambda\vdash n}a_{\lambda}d_{\lambda}
\]
and 
\[
\dim_{F}SI_{n}=\sum_{\lambda\vdash n}a{}_{\lambda}^{\prime}d_{\lambda}.
\]

Moreover, by \corref{m=00003Dm'} 
\[
\sum_{\begin{array}{c}
\lambda\vdash n\\
\lambda_{m^{2}}>1
\end{array}}a{}_{\lambda}^{\prime}d_{\lambda}=\sum_{\begin{array}{c}
\lambda\vdash n\\
\lambda_{m^{2}}>1
\end{array}}a_{\lambda}d_{\lambda}=a_{n}
\]
and 
\[
b{}_{n}^{\prime}=\sum_{\begin{array}{c}
\lambda\vdash n\\
\lambda_{m^{2}}\leq1
\end{array}}a{}_{\lambda}^{\prime}d_{\lambda}\leq\sum_{\begin{array}{c}
\lambda\vdash n\\
\lambda_{m^{2}}\leq1
\end{array}}a_{\lambda}d_{\lambda}=b_{n},
\]
i.e. $\dim_{F}I_{n}=a_{n}+b_{n}$ and $\dim_{F}SI_{n}=a_{n}+b{}_{n}^{\prime}$.

If we confirm that $b_{n}$ is negligible by showing 
\[
\mbox{\ensuremath{\exp}}\left(b_{n}\right)<\exp(a_{n}+b_{n})=\exp\left(I_{n}\right),
\]
then it will follow that 
\[
\exp\left(b{}_{n}^{\prime}\right)\leq\exp\left(b_{n}\right)<\exp(a_{n})\leq\exp(a_{n}+b{}_{n}^{\prime}).
\]
Hence, $b{}_{n}^{\prime}$ is also negligible, and the theorem follows. 

Indeed, by \cite{key-10} theorem 5.10.3, $d_{\lambda}\leq n^{m^{2}}(m^{2}-1)^{n}$,
where $\lambda\vdash n$ and $\lambda_{m^{2}}\leq1$. Thus by \lemref{m-poly},
\[
b_{n}=\sum_{\begin{array}{c}
\lambda\vdash n\\
\lambda_{m^{2}}\leq1
\end{array}}a_{\lambda}d_{\lambda}\leq p(n)n^{m^{2}}(m^{2}-1)^{n}.
\]
So, $\mbox{\ensuremath{\exp}}\left(b_{n}\right)\leq m^{2}-1$, and
on the other hand, by \corref{codim-asymp}, $\exp(I_{n})=m^{2}$.
\end{proof}

\subsection{The final step}

Recall that \thmref{.glory} is true only when the group $\widetilde{G}$
is connected. We intend to remove this assumption, leaving no restrictions
on $\widetilde{G}$. First we recall some notations we need from \corref{We-have-where}: 

The group $H_{\mathfrak{g}}$ acts on $B$ by left multiplication.
Let $B'\subseteq B$ be a set of representatives of the orbits (of
the above action), and 
\[
I{}_{n}^{\prime}=\sp\left\{ T{}_{\sigma,\mathbf{h}}^{\prime}|\sigma\in S_{n},\mathbf{h}\in B^{n}\right\} =\bigoplus_{g\in H_{\mathfrak{g}}}L_{g},
\]
 where $L_{g}=\sp\left\{ T{}_{\sigma,\mathbf{h}}^{\prime}\in I{}_{n}^{\prime}|h_{1}\in gB'\right\} $.

In the same spirit, denote 
\[
SI{}_{n}^{\prime}=\sp\left\{ T{}_{\sigma,\mathbf{h}}^{\prime}|\sigma\mbox{ is a cycle},\mathbf{h}\in B^{n}\right\} =\bigoplus_{g\in H_{\mathfrak{g}}}SL_{g},
\]
 where $SL_{g}=\sp\left\{ T{}_{\sigma,\mathbf{h}}^{\prime}\in SI{}_{n}^{\prime}|h_{1}\in gB'\right\} $.

Let $\overline{G}^{\circ}$ be the identity component of $\overline{G}$,
clearly $\overline{G}^{\circ}=GL_{m_{1}}(F)\times\cdots\times GL_{m_{k}}(F)$.
Observe that $T_{n}^{\overline{G}^{\circ}}=I{}_{n}^{\prime}$. So,
by \thmref{.glory}, $\dim_{F}I{}_{n}^{\prime}\sim\dim_{F}SI{}_{n}^{\prime}$.
As in \corref{We-have-where}, we find that $T_{n}^{\widetilde{G}}=I_{n}\cong L_{g}$
and $SI_{n}\cong SL_{g}$ for every $g\in H_{\mathfrak{g}}$. Thus,
\[
\dim_{F}I_{n}=\frac{1}{|H_{\mathfrak{g}}|}\dim_{F}I{}_{n}^{\prime}\sim\frac{1}{|H_{\mathfrak{g}}|}\dim_{F}SI{}_{n}^{\prime}=\dim_{F}SI_{n}.
\]

\section{\label{sec:SI=00003DC} The Codimension Sequence Asymptotics of a
Matrix Algebra with Elementary Grading}

In this section we will prove that the spaces of special invariants
approximate the codimension sequence. More precisely, we define a
subspace $\Tr(C_{n}^{G})$ of $SI_{n}$ with dimension $c_{n-1}^{G}(A)$,
and show that it is dominant in $SI_{n}$ (that is, $\Tr(C_{n}^{G})\sim SI_{n}$).
Combining this with the previous parts of this article shows $c_{n-1}^{G}(A)\sim t_{n}^{G}(A)$.
Thus, by \corref{codim-asymp}, the asymptotics of $c_{n}^{G}(A)$
will be established. To present $\Tr(C_{n}^{G})$ we need the following
definitions.

Let 
\[
R_{n}^{G}(A)=F\left\langle U_{g}^{(r)}=\sum_{t,tg\in B;i,j}u_{t(i),tg(j)}^{(r)}e_{t(i),tg(j)}\,|\, g\in B^{-1}B\,\,1\leq r\leq n\right\rangle 
\]
 be the ring of $G$-graded $m\times m$ generic matrices in the variables
$u_{s(i),t(j)}^{(r)}$, where $s,t\in B$, $1\leq i\leq m_{s}$, $1\leq j\leq m_{t}$
and $1\leq r\leq n$. It is well known that $R_{n}^{G}(A)$ is isomorphic
to the $G$-graded relatively free algebra i.e. 
\[
R_{n}^{G}(A)\backsimeq\nicefrac{F\left\langle x_{1,g_{1}},...,x_{n,g_{1}},...,x_{1,g_{k}},...,x_{n,g_{k}}\right\rangle }{Id^{G}(A)},
\]
so the subspace of multilinear elements in $R_{n}^{G}(A)$ is

\[
C_{n}^{G}=\sp\left\{ U_{h_{\sigma(1)}}^{\left(\sigma(1)\right)}\cdots U_{h_{\sigma(n)}}^{\left(\sigma(n)\right)}|\sigma\in S_{n}\,,\, h_{i}\in G\right\} \subset R_{n}^{G}(A).
\]
Obviously its dimension is $c_{n}^{G}(A)$.

\subsection{The space of traces}

Let $tr:M_{m}(F)\to F$ be the standard trace function. We are interested
in the spaces $\Tr(C_{n}^{G})=\left\{ tr(x)\,|\, x\in C_{n}^{G}\right\} $
which, as we will show, have dimension $c_{n-1}^{G}(A)$, and dominant
in $SI_{n}$. 
\begin{lem}
$\dim_{F}\Tr(C_{n}^{G})=c_{n-1}^{G}(A)$. \end{lem}
\begin{proof}
$tr$ surely satisfies all the conditions of \lemref{trace}.
\end{proof}
Next we intend to demonstrate that $t_{n}^{G}(A)=\dim_{F}SI_{n}\sim\dim_{F}\Tr(C_{n}^{G})=c_{n-1}^{G}(A)$.
The plan is to define yet another two subspaces of $SI_{n}$ ($CSI_{n}$
and $COSI_{n}$) such that

\begin{equation}
\xymatrix{\Tr(C_{n}^{G})\ar@{^{(}->}[r] & SI_{n}\\
COSI_{n}\ar@{^{(}->}[r]\ar@{^{(}->}[u] & CSI_{n}\ar@{^{(}->}[u]
}
\label{eq:diagram}
\end{equation}
By proving that $CSI_{n}$ is dominant in $SI_{n}$, and that $COSI_{n}$
is dominant in $CSI_{n}$ the desired result, i.e. $\dim_{F}\Tr(C_{n}^{G})\sim\dim_{F}SI_{n}$,
will hold.
\begin{rem}
\label{rem:T-tag}Recall that 
\[
\pi:T_{n}\longrightarrow\End_{F}(V^{\otimes n})
\]
\[
\pi\left(u_{s_{1}(i_{1}),t_{1}(j_{1})}^{(1)}\cdots u_{s_{n}(i_{n}),t_{n}(j_{n})}^{(n)}\right)=T{}_{\mathbf{s(i)},\mathbf{t(j)}}^{\prime},
\]
 where $T{}_{\mathbf{s(i)},\mathbf{t(j)}}^{\prime}(e_{\mathbf{h(k)}})=\delta_{\mathbf{h(k)},\mathbf{s(i)}}e_{\mathbf{t(j)}}$,
and notice that by definition 
\[
T{}_{\sigma,\mathbf{g}}^{\prime}=\sum_{\mathbf{i\in}M_{\mathbf{g}}}T{}_{\mathbf{g(i),g^{\sigma}(i^{\sigma})}}^{\prime},
\]
 where $M_{\mathbf{g}}=\{1,..,m_{g_{1}}\}\times\cdots\times\{1,...,m_{g_{n}}\}$.
\end{rem}
Let's start by showing that via the isomorphism $\pi$ the space $\Tr(C_{n}^{G})$
is embedded in $SI_{n}$.
\begin{lem}
\textup{\label{lem:TR-in-SI}Let $Z=U_{h_{\sigma(1)}}^{\left(\sigma(1)\right)}\cdots U_{h_{\sigma(n)}}^{\left(\sigma(n)\right)}$
be a monomial such that $0\ne tr(Z)$ (in particular, $Z\neq0$).
Then there is $\mathbf{\tilde{h}=}(\tilde{h}_{1},..,\tilde{h}_{n})\in B^{n}$
such that
\[
\pi\left(tr(X)\right)=\sum_{t\in\left[H_{\mathbf{\tilde{h}}}:H_{\mathfrak{g}}\right]}T_{\tau,t\mathbf{\tilde{h}}},
\]
where }$H_{\mathbf{\tilde{h}}}=\left\{ t\in B|\{t\tilde{h}_{1},...,t\tilde{h}_{n}\}\subseteq B\right\} $
and \textup{$\left[H_{\mathbf{\tilde{h}}}:H_{\mathfrak{g}}\right]$
is a full representative set of the left $H_{\mathfrak{g}}$ cosets
of $H_{\mathbf{\tilde{h}}}$}. \end{lem}
\begin{proof}
For simplicity denote $\sigma(i)=\sigma_{i}$. If $h_{\sigma_{1}}\cdots h_{\sigma_{n}}\ne e$
then the trace will be zero. So, we may assume that $h_{\sigma_{1}}\cdots h_{\sigma_{n}}=e$. 

Let $trZ_{t,t}$ be the trace of the $(t,t)$-block of $X$, and define
$\mathbf{\tilde{h}=}(\tilde{h}_{1},..,\tilde{h}_{n})$ by $\tilde{h}_{\sigma_{i}}=h_{\sigma_{1}}\cdots h_{\sigma_{i-1}}$
and $\tilde{h}_{\sigma_{1}}=t$, such that $tr(Z)_{t,t}\neq0$. Then
for any $t\in H_{\tilde{\mathbf{h}}}$, 
\[
tr\left(U_{h_{\sigma_{1}}}^{(\sigma_{1})}\cdots U_{h_{\sigma_{n}}}^{(\sigma_{n})}\right)_{t,t}=\sum_{\mathbf{i}\in M_{\mathbf{\tilde{h}}}}u_{t(i_{\sigma_{1}}),th_{\sigma_{1}}(i_{\sigma_{2}})}^{(\sigma_{1})}u_{th_{\sigma_{1}}(i_{\sigma_{2}}),th_{\sigma_{1}}h_{\sigma_{2}}(i_{\sigma_{3}})}^{(\sigma_{2})}\cdots u_{th_{\sigma_{1}}\cdots h_{\sigma_{n-1}}(i_{\sigma_{n}}),t(i_{\sigma_{1}})}^{(\sigma_{n})}
\]
otherwise $tr\left(U_{h_{\sigma_{1}}}^{(\sigma_{1})}\cdots U_{h_{\sigma_{n}}}^{(\sigma_{n})}\right)_{t,t}=0$
(since $t\notin B$ or one of the $th_{\sigma_{1}}\cdots h_{\sigma_{i}}=t\tilde{h}_{i+1}$
is not in $B$). By taking the $n$-cycle $\tau=(\sigma_{1},...,\sigma_{n})$
in $S_{n}$ we can rewrite the expression above 

\[
tr\left(U_{h_{\sigma_{1}}}^{(\sigma_{1})}\cdots U_{h_{\sigma_{n}}}^{(\sigma_{n})}\right)_{t,t}=\sum_{\mathbf{i}\in M_{\mathbf{\tilde{h}}}}u_{t\tilde{h}_{\sigma_{1}}(i_{\sigma_{1}}),t\tilde{h}_{\tau(\sigma_{1})}(i_{\tau(\sigma_{1})})}^{(\sigma_{1})}\cdots u_{t\tilde{h}_{\sigma_{n}}(i_{\sigma_{n}}),t\tilde{h}_{\tau(\sigma_{n})}(i_{\tau(\sigma_{n})})}^{(\sigma_{n})}.
\]
Reordering the monomial results in the simpler expression: 
\[
tr\left(U_{h_{\sigma_{1}}}^{(\sigma_{1})}\cdots U_{h_{\sigma_{1}}}^{(\sigma_{n})}\right)_{t,t}=\sum_{\mathbf{i}\in M_{\mathbf{\tilde{h}}}}u_{t\tilde{h}_{1}(i_{1}),t\tilde{h}_{\tau(1)}(i_{\tau(1)})}^{(1)}\cdots u_{t\tilde{h}_{n}(i_{n}),t\tilde{h}_{\tau(n)}(i_{\tau(n)})}^{(n)}.
\]
By applying $\pi$ on this equation together with \remref{T-tag}
we obtain
\[
\pi\left(tr\left(U_{h_{\sigma_{1}}}^{(\sigma_{1})}\cdots U_{h_{\sigma_{1}}}^{(\sigma_{n})}\right)_{t,t}\right)=\sum_{\mathbf{i}\in M_{\mathbf{\tilde{h}}}}T{}_{t\mathbf{\tilde{h}(i)},t\mathbf{\tilde{h}^{\tau}(i^{\tau})}}^{\prime}=T{}_{\tau,t\mathbf{\tilde{h}}}^{\prime}.
\]
 Summing it up gives us 
\[
\pi\left(tr\left(U_{h_{\sigma_{1}}}^{(\sigma_{1})}\cdots U_{h_{\sigma_{1}}}^{(\sigma_{n})}\right)\right)=\sum_{t\in H_{\mathbf{\tilde{h}}}}T{}_{\tau,t\mathbf{\tilde{h}}}^{\prime}.
\]

Finally, note that $H_{\mathfrak{g}}$\textbf{ }is a subgroup of \textbf{$H_{\mathbf{\tilde{h}}}$},
and $T_{\tau,\mathbf{g}}=\sum_{g\in H_{\mathfrak{g}}}T{}_{\tau,g\mathbf{g}}^{\prime}$.
Thus,
\[
\pi\left(tr\left(U_{h_{\sigma_{1}}}^{(\sigma_{1})}\cdots U_{h_{\sigma_{1}}}^{(\sigma_{n})}\right)\right)=\sum_{t\in\left[H_{\mathbf{\tilde{h}}}:H_{\mathfrak{g}}\right]}T_{\tau,t\mathbf{\tilde{h}}}.
\]
\end{proof}
\begin{cor}
\textup{\label{cor:TrinSI}}$\Tr(C_{n}^{G})$ is embedded in $SI_{n}$.
\end{cor}

\subsection{Complete invariants}

We now approach the definition of the first mentioned subspace of
$SI_{n}$: the space of \textit{complete special invariants} $CSI_{n}$. 
\begin{defn}
A vector $\mathbf{g}\in B^{n}$ is \emph{complete} if all the elements
of $B$ appear in $\mathbf{g}$. Moreover, denote $CSI_{n}=\sp\left\{ T_{\sigma,\mathbf{g}}\in SI_{n}|\mathbf{g}\mbox{ is complete}\right\} $. 
\end{defn}
We claim that this space is dominant in $SI_{n}$. In other words:
\begin{lem}
\label{lem:SI=00003DCSI}$\dim_{F}SI_{n}\sim\dim_{F}CSI_{n}$\end{lem}
\begin{proof}
It is enough to show that $\exp\left(NCSI_{n}\right)<\exp\left(SI_{n}\right)(=\exp\left(I_{n}\right)=m^{2})$,
where $NCSI_{n}=\sp\left\{ T_{\sigma,\mathbf{g}}\in SI_{n}|\mathbf{g}\mbox{ isn't complete}\right\} $,
since $\dim_{F}SI_{n}=\dim_{F}CSI_{n}+\dim_{F}NCSI_{n}$. 

Consider the spaces $NCI_{n}=\sp\left\{ T_{\sigma,\mathbf{g}}\in I_{n}|\mathbf{g}\mbox{ isn't complete}\right\} $.
Because $NCSI_{n}\subseteq NCI_{n}$, it suffices to establish: 
\[
\exp\left(NCI_{n}\right)<m^{2}.
\]
Notice we can write $NCI_{n}=\sum_{g\in B}NCI_{n}(g)$, where 
\[
NCI_{n}(g)=\sp\left\{ T_{\sigma,\mathbf{g}}\in I_{n}|g\mbox{ doesn't appear in }\mathbf{g}\right\} .
\]
It is, therefore, enough to prove that $\exp\left(NCI_{n}(g)\right)<m^{2}$
for any $g\in B$. To show this we must recall some of the notation
of \secref{Invariants}: 

The operators $T{}_{\sigma,\mathbf{g}}^{\prime}(A)=T{}_{\sigma,\mathbf{g}}^{\prime}$
are elements in $\End_{F}(V^{\otimes n})$, where $V=\oplus_{g\in B}V_{g}$
and $\dim_{F}V_{g}=m_{g}$. Fix $g\in G$, and consider the $G$-graded
algebra $\bar{A}=M_{m-m_{g}}(F)$ with elementary grading corresponding
to the grading vector $\mathfrak{g'}$ which is obtained from $\mathfrak{g}$
by erasing all the $g$'s. For this algebra $T'_{\sigma,\mathbf{g}}(\bar{A})$
are elements in $\End\left(\bar{V}^{\otimes n}\right)$ where $\bar{V}=\oplus_{g\ne h\in B}V_{h}$,
so it is easy to see that the mapping 
\[
I{}_{n}^{\prime}(\bar{A})=\sp\left\{ T{}_{\sigma,\mathbf{g}}^{\prime}(\bar{A})\right\} \rightarrow\sp\left\{ T{}_{\sigma,\mathbf{g}}^{\prime}(A)|g\mbox{ doesn't appear in \textbf{\ensuremath{\mathbf{g}}}}\right\} :=\left(I'\right)_{n}^{(g)}(A)
\]
taking $T{}_{\sigma,\mathbf{g}}^{\prime}(\bar{A})$ to $T{}_{\sigma,\mathbf{g}}^{\prime}(A)$,
is an isomorphism of vector spaces. Moreover, since $m_{t}=m_{ht}$
for every $h\in H_{\mathfrak{g}}$ and $t\in B$, the mapping 
\[
\left(I'\right)_{n}^{(g)}(A)\rightarrow\left(I'\right)_{n}^{(hg)}(A)
\]
taking $T{}_{\sigma,\mathbf{g}}^{\prime}(A)$ to $T{}_{\sigma,h\mathbf{g}}^{\prime}(A)$
is also an isomorphism of vector spaces for every $h\in H_{\mathfrak{g}}$.
Hence, by virtue of \corref{codim-asymp} and \corref{We-have-where},
\[
\exp\left(\left(I'\right)_{n}^{(g)}(A)\right)=\exp\left(\left(I'\right)_{n}^{(hg)}(A)\right)=\exp\left(I{}_{n}^{\prime}(\bar{A})\right)=\exp\left(t_{n}^{G}(\bar{A})\right)=(m-m_{g})^{2}.
\]
Since, 
\[
NCI_{n}(g)=\sp\left\{ T_{\sigma,\mathbf{g}}(A)=\sum_{h\in H_{\mathfrak{g}}}T{}_{\sigma,h\mathbf{g}}^{\prime}(A)|g\mbox{ doesn't appear in }\mathbf{g}\right\} \subset\sum_{h\in H_{\mathfrak{g}}}\left(I'\right)_{n}^{(hg)}(A),
\]
 $\exp\left(NCI_{n}(g)\right)=(m-m_{g})^{2}<m^{2}$.
\end{proof}

\subsection{In-order invariants}

To define the second mentioned subspace of $SI_{n}$ we first consider
the decomposition $B=\bigsqcup_{i=1}^{s}B_{i}$ where $B_{i}$ consists
of all the $g$'s in $B$ having a common value $b_{i}=m_{g}$ for
all $g\in B_{i}$. We assume that $b_{1}<b_{2}<\cdots<b_{s}$ and
recall that $m_{1}\leq\cdots\leq m_{k}$. Next, for every vector $\mathbf{g}\in B^{n}$
we define $n_{\gamma_{i}}(\mathbf{g})=\#\mbox{ appearances of \ensuremath{\gamma_{i}} in \ensuremath{\mathbf{g}}}$.
A vector $\mathbf{g}$ is called\textit{ in-order} if $n_{h_{1}}(\mathbf{g})<\cdots<n_{h_{s}}(\mathbf{g})$
for all $h_{1}\in B_{1},...,h_{s}\in B_{s}$. 
\begin{example}
Suppose $G=\mathbb{Z}_{3}=\{0,1,2\}$ and the grading vector $\mathfrak{g}=(2,0,0,0,1,1,1)$.
In this case, $B=B_{1}\cup B_{2}$, where $B_{1}=\{2\}$ and $B_{2}=\{0,1\}$.
Moreover, $b_{1}=1$ and $b_{2}=3$. The vector $\mathbf{g}=(2,2,1,1,2,2,0,0,0,1,0)$
is not in-order, since $n_{2}(\mathbf{g})=4$ but $n_{1}(\mathbf{g})=3$;
the vector $\mathbf{h}=(0,1,1,1)$ is in-order, but not complete,
since $2\in B$ is missing. Finally, the vector $\mathbf{s}=(1,2,0,0,1)$
is in-order and complete.
\end{example}
Finally denote 
\[
COSI_{n}=\sp\left\{ T_{\sigma,\mathbf{g}}\in CSI_{n}|\mathbf{g}\mbox{ is in-order}\right\} .
\]
We claim that these spaces are dominant in $CSI_{n}$, namely:
\begin{lem}
\textup{$\dim_{F}CSI_{n}\sim\dim_{F}COSI_{n}$.}\end{lem}
\begin{proof}
As previously demonstrated, it is sufficient to show that $\exp\left(CNOSI_{n}\right)<m^{2}=\exp\left(CSI_{n}\right)$,
where $CNOSI_{n}=\sp\left\{ T_{\sigma,\mathbf{g}}\in CSI_{n}|\mathbf{g}\mbox{ is not in-order}\right\} $.
Consider the space $NOI_{n}$ which is the span of all the $T_{\sigma,\mathbf{g}}\in I_{n}$
where $\mathbf{g}$ is not in-order. Because $CNOSI_{n}\subseteq NOI_{n}$
it is sufficient to establish 
\[
\exp\left(NOI_{n}\right)<m^{2}.
\]
To accomplish this we use \remref{For-future-use} to obtain:\foreignlanguage{american}{
\[
\dim_{F}\left(NOI_{n}\right)=\frac{1}{|H_{\mathfrak{g}}|}\sum_{\sigma\in P<S_{k}}\sum_{\begin{array}{c}
n_{1}+\cdots+n_{k}=n\\
n_{1}\leq\cdots\leq n_{k}
\end{array}}{n \choose n_{1},...,n_{k}}^{2}t_{n_{1}}\left(m_{1}\right)\cdots t_{n_{k}}\left(m_{k}\right)
\]
where $P$ }consists of $\sigma\in S_{k}$ having the property that
there exist $i,j\in\{1,...,k\}$ such that $m_{\sigma(j)}<m_{\sigma(i)}$
but $n_{j}=n_{g_{\sigma(j)}}(\mathbf{g})>n_{g_{\sigma(i)}}(\mathbf{g})=n_{i}$\foreignlanguage{american}{.
This is less or equal to (see \propref{Calc}):}

\selectlanguage{american}%
\[
\frac{1}{|H_{\mathfrak{g}}|}\sum_{\sigma\in P}\sum_{\begin{array}{c}
n_{1}+\cdots+n_{k}=n\\
n_{1}\leq\cdots\leq n_{k}
\end{array}}{n \choose n_{1},...,n_{k}}^{2}m_{\sigma(1)}^{2n_{1}}\cdots m_{\sigma(k)}^{2n_{k}}.
\]
\foreignlanguage{english}{So it is enough to confirm that for any
$\sigma\in P$ the exponent of 
\[
S(\sigma)=\sum_{\begin{array}{c}
n_{1}+\cdots+n_{k}=n\\
n_{1}\leq\cdots\leq n_{k}
\end{array}}{n \choose n_{1},...,n_{k}}^{2}m_{\sigma(1)}^{2n_{1}}\cdots m_{\sigma(k)}^{2n_{k}}
\]
is smaller than $m^{2}$. }

\selectlanguage{english}%
Fix such $\sigma$ together with $i$ and $j$ coming from the definition
of $P$ (i.e. $m_{\sigma(j)}<m_{\sigma(i)}$ and $n_{j}>n_{i}$).
Let $1>\theta>0$ be some real constant which will be determined later
on. Split the above sum into two: {\scriptsize{}
\[
\sum_{\begin{array}{c}
n_{1}+\cdots+n_{k}=n\\
\begin{array}{c}
n_{1}\leq\cdots\leq n_{k}\\
n_{j}-n_{i}\geq\theta n
\end{array}
\end{array}}{n \choose n_{1},...,n_{k}}^{2}m_{\sigma(1)}^{2n_{1}}\cdots m_{\sigma(k)}^{2n_{k}}+\sum_{\begin{array}{c}
n_{1}+\cdots+n_{k}=n\\
\begin{array}{c}
n_{1}\leq\cdots\leq n_{k}\\
n_{j}-n_{i}<\theta n
\end{array}
\end{array}}{n \choose n_{1},...,n_{k}}^{2}m_{\sigma(1)}^{2n_{1}}\cdots m_{\sigma(k)}^{2n_{k}}.
\]
}Call the first sum $\mathcal{S}_{1}$ and the other $\mathcal{S}_{2}$.
We start with estimating $S_{1}$: For $\alpha=\frac{m_{\sigma(j)}}{m_{\sigma(i)}}<1$
and $\lambda=(i\, j)\sigma\in S_{n}$ we have {\footnotesize{}
\[
\mathcal{S}_{1}=\sum_{\begin{array}{c}
n_{1}+\cdots+n_{k}=n\\
\begin{array}{c}
n_{1}\leq\cdots\leq n_{k}\\
n_{j}-n_{i}\geq\theta n
\end{array}
\end{array}}\alpha^{2(n_{j}-n_{i})}{n \choose n_{1},...,n_{k}}^{2}m_{\lambda(1)}^{2n_{1}}\cdots m_{\lambda(k)}^{2n_{k}}\leq\alpha^{2\theta n}S(\lambda)\leq\alpha^{2\theta n}dim_{F}I_{n}.
\]
}Thus $\exp(S_{1})\leq\alpha^{2\theta}m^{2}<m^{2}$. (We used \corref{codim-asymp}
in the last step).

Next we turn to estimating $\mathcal{S}_{2}$. To simplify notation
assume, without loss of generality, that $i=1$ and $j=2$, and write
$n_{2}=n_{1}+z$. We have: 
\[
S_{2}\leq\sum_{\begin{array}{c}
0\leq z<\theta n\end{array}}\sum_{\begin{array}{c}
2n_{1}+n_{3}+\cdots+n_{k}=n-z\\
n_{1}\leq n_{3}\leq\cdots\leq n_{k}
\end{array}}{n \choose n_{1},...,n_{k}}^{2}m_{\sigma(1)}^{2n_{1}}\cdots m_{\sigma(k)}^{2n_{k}}.
\]
Notice that 
\[
{n \choose n_{1},...,n_{k}}={n-z \choose n_{1},n_{1},n_{3},...,n_{k}}\frac{n\cdots(n-z+1)}{(n_{1}+1)\cdots(n_{1}+z)}\leq{n-z \choose n_{1},n_{1},n_{3},...,n_{k}}\frac{n^{z}}{z!}.
\]
Furthermore, 
\[
{n-z \choose n_{1},n_{1},n_{3},...,n_{k}}={n-z \choose 2n_{1},n_{3},...,n_{k}}{2n_{1} \choose n_{1}}\leq2^{2n_{1}}{n-z \choose 2n_{1},n_{3},...,n_{k}}.
\]
Combining all of these together we conclude that $\mathcal{S}_{2}$
is less than or equal to{\scriptsize{}
\[
\sum_{\begin{array}{c}
0\leq z<\theta n\end{array}}\left(\frac{n^{z}}{z!}\right)^{2}\left(\sum_{\begin{array}{c}
2n_{1}+n_{3}+\cdots+n_{k}=n-z\\
n_{1}\leq n_{3}\leq\cdots\leq n_{k}
\end{array}}2^{4n_{1}}{n-z \choose 2n_{1},n_{3},...,n_{k}}^{2}m_{\sigma(1)}^{2n_{1}}\cdots m_{\sigma(k)}^{2n_{k}}\right).
\]
}The expression inside the (big) parentheses is $\leq${\scriptsize{}
\[
m_{\sigma(2)}^{2\theta n}\cdot\sum_{\begin{array}{c}
2n_{1}+n_{3}+\cdots+n_{k}=n-z\\
n_{1}\leq n_{3}\leq\cdots\leq n_{k}
\end{array}}2^{4n_{1}}{n-z \choose 2n_{1},n_{3},...,n_{k}}^{2}\left(\sqrt{m_{\sigma(1)}m_{\sigma(2)}}\right)^{^{4n_{1}}}m_{\sigma(3)}^{2n_{3}}\cdots m_{\sigma(k)}^{2n_{k}}.
\]
}If we change the names of the indexes by $\tilde{n}_{1}\leftarrow2n_{1},\tilde{n}_{2}\leftarrow n_{3},...,\tilde{n}_{k-1}\leftarrow n_{k}$
and write $\tilde{m}_{1}=2\cdot\sqrt{m_{\sigma(1)}m_{\sigma(2)}},\tilde{m}_{2}=m_{\sigma(3)},...,\tilde{m}_{k-1}=m_{\sigma(k)}$
we get
\[
\leq m_{\sigma(2)}^{2\theta n}\cdot\sum_{\tilde{n}_{1}+\cdots+\tilde{n}_{k-1}=n-z}{n-z \choose \tilde{n}_{1},...,\tilde{n}_{k-1}}^{2}\tilde{m}_{1}^{2\tilde{n}_{1}}\cdots\tilde{m}_{k-1}^{2\tilde{n}_{k-1}}
\]
Hence, using a similar argument to that in the proof of \corref{codim-asymp},
we have for some constant $C>0$:
\[
\leq Cm_{\sigma(2)}^{2\theta n}\left(\tilde{m}_{1}+\cdots+\tilde{m}_{k-1}\right)^{2n}.
\]
All in all, $\mathcal{S}_{2}\leq Cn^{\epsilon}\left(\left(\frac{m_{\sigma(2)}e}{\theta}\right)^{2\theta}\right)^{n}\left(\tilde{m}_{1}+\cdots+\tilde{m}_{k-1}\right)^{2n}$
(we used here Stirling's approximation). 

Since $\left(\tilde{m}_{1}+\cdots+\tilde{m}_{k-1}\right)^{2}<m^{2}$
(this is due to $2\sqrt{xy}<x+y$ when $x\neq y$) and $\underset{\theta\rightarrow0^{+}}{\lim}\left(\frac{m_{\sigma(2)}e}{\theta}\right)^{2\theta}=1$,
we can take small enough $\theta$ to obtain: 
\[
Cn^{\epsilon}\left(\left(\frac{m_{\sigma(2)}e}{\theta}\right)^{2\theta}\right)^{n}\left(\tilde{m}_{1}+\cdots+\tilde{m}_{k-1}\right)^{2n}<m^{2n},
\]
yielding $\exp(\mathcal{S}_{2})<m^{2}$ and $\exp(S(\sigma))<m^{2}$.
\end{proof}

\subsection{The final step}

We are ready to show that $COSI_{n}$ embeds inside $\Tr\left(C_{n}\right)$. 
\begin{lem}
If $\mathbf{g}=(g_{1},...,g_{n})$ is complete and in-order then $g\mathbf{g}$
is not in-order for every non trivial $g\in[H_{B}:H_{\mathfrak{g}}]$
(see \defref{sub_groups_of_G}).\end{lem}
\begin{proof}
Let $g$ be non-trivial in $\left[H_{B}:H_{\mathfrak{g}}\right]$.
Since $g\notin H_{\mathfrak{g}}$ and $\mathbf{g}$ is complete, there
are $g_{i'}\in B_{i}$ and $g_{j'}\in B_{j}$, such that $gg_{i'}=g_{j'}$
but $m_{g_{i'}}\neq m_{g_{j'}}$, or put differently, $gB_{i}\cap B_{j}\neq\emptyset$.
Suppose without loss of generality that $m_{i}>m_{j}$ (so, $i>j$).
Let $i$ be the maximal index we can find having the above property.
Namely, the maximal $i$ for which there is an index $j<i$ such that
$gB_{i}\cap B_{j}\neq\emptyset$. 

We claim that there is an index $r<i$ for which $B_{i}\cap gB_{r}\neq\emptyset$.
Indeed, since $B_{i}\cap(gB=B)\neq\emptyset$ and $B=\cup B_{l}$,
there must be an $r$ such that $B_{i}\cap gB_{r}\neq\emptyset$.
However, due to the maximality of $i$, $r\leq i$. Furthermore, $gB_{i}\cap B_{j}\neq\emptyset$
yields $|gB_{i}\cap B_{i}|<|B_{i}|$. So, $r<i$ as required.

Finally, take $x\in gB_{i}\cap B_{j}$ and $y\in B_{i}\cap gB_{r}$.
Then, $n_{g^{-1}y}(g\mathbf{g})=n_{y}(\mathbf{g})>n_{x}(\mathbf{g})=n_{g^{-1}x}(g\mathbf{g})$
(since $x\in B_{j}$, $y\in B_{i}$ and $j<i$) . However, $g^{-1}y\in B_{r},$
$g^{-1}x\in B_{i}$ and $r<i$ so $g\mathbf{g}$ is not in-order.\end{proof}
\begin{cor}
$COSI_{n}$ is embedded in $\Tr(C_{n}^{G})$.\end{cor}
\begin{proof}
We define the $F$-vector spaces mapping $\psi:COSI_{n}\rightarrow\Tr(C_{n}^{G})$
by 
\[
\psi(T_{\sigma,\mathbf{g}})=\pi^{-1}\left(\sum_{h\in[H_{B}:H_{\mathfrak{g}}]}T{}_{\sigma,h\mathbf{g}}\right).
\]
Note that since $\mathbf{g}$ is complete $H_{B}=H_{\mathbf{g}}$.
So, by \lemref{TR-in-SI}, $\psi(T_{\sigma,\mathbf{g}})$ is indeed
in $Tr(C_{n}^{G})$. 

If $\psi\left(\sum_{i}a_{i}T_{\sigma_{i},\mathbf{g_{\mathbf{i}}}}\right)=0$,
then 
\[
\sum_{i}a_{i}T{}_{\sigma_{i},\mathbf{g}_{i}}+\sum_{i}\sum_{e\neq h\in[H_{B}:H_{\mathfrak{g}}]}a_{i}T{}_{\sigma_{i},h\mathbf{g}_{i}}=0
\]
We know that all the vectors $h\mathbf{g_{\mathbf{i}}}$ are not in-order
while all the $\mathbf{g}_{i}$'s are in-order. Thus, the linear span
of $T{}_{\sigma_{i},\mathbf{g_{\mathbf{i}}}}$ is linearly independent
of the span of $\left\{ T{}_{\sigma_{i},h\mathbf{g_{\mathbf{i}}}}|e\neq h\in[H_{B}:H_{\mathfrak{g}}]\right\} $.
Therefore, $\sum_{i}a_{i}T{}_{\sigma_{i},\mathbf{g_{\mathbf{i}}}}=0$
as claimed.
\end{proof}
We proved that $\dim_{F}COSI_{n}\sim\dim_{F}(I_{n})$ and showed that
all the maps in (\ref{eq:diagram}) are indeed embeddings, so 
\[
\xymatrix{COSI_{n}\ar@{^{(}->}[r] & \Tr(C_{n}^{G})\ar@{^{(}->}[r] & I_{n}}
.
\]
Hence 
\[
c_{n}^{G}(A)=\dim_{F}\Tr(C_{n+1}^{G})\sim\dim_{F}I_{n+1}=t_{n+1}(A).
\]
Therefore, by \corref{We-have-where} we obtain the main result of
this article:
\begin{thm*}[\textbf{A}]
\label{thm:A} Let $G$ be a finite group, and $F$ be a field of
characteristic zero. Suppose $A$ is the $F$-algebra of $m\times m$
matrices with elementary $G$-grading defined by the grading vector
$\mathfrak{g}=(\gamma_{1}^{m_{1}},...,\gamma_{k}^{m_{k}})$, where
all the $\gamma_{i}$'s are distinct. Then 
\[
c_{n}^{G}(A)\sim\alpha n^{\frac{1-\left(\sum_{i=1}^{k}m_{i}^{2}\right)}{2}}m^{2n}=\alpha n^{\frac{1-\dim_{F}A_{e}}{2}}(\dim_{F}A)^{n}
\]
where
\[
\alpha=\frac{1}{|H_{\mathfrak{g}}|}m^{\frac{\sum_{i=1}^{k}m_{i}^{2}}{2}+2}\left(\frac{1}{\sqrt{2\pi}}\right)^{m-1}\left(\frac{1}{2}\right)^{\frac{\sum_{i=1}^{k}m_{i}^{2}-1}{2}}\prod_{i=1}^{k}\left(1!2!\cdots(m_{i}-1)!m_{i}^{-\frac{1}{2}}\right).
\]

\end{thm*}

\section{\label{sec:Fine}Fine grading case}

Recall from the introduction that the asymptotics of the codimension
sequence of any affine $G$-graded algebra has the form $\alpha n^{\beta}d^{n}$.
In the previous sections we were successful in finding the constants
$\alpha,\beta$ and $d$ for the case of matrix algebra with elementary
grading. We intend to push forward this result. Namely, to calculate
the constants $\beta$ and $d$ for $G$-simple finite dimensional
algebras (recall that elementary grading is a $G$-simple grading).

As previously noted (see \thmref{Let--be}), the $G$-simple algebras
are composed of two parts: the elementary part, which we discussed
thoroughly and the fine part. Although the asymptotic behavior of
the codimension sequence of fine graded algebras is known (see \cite{key-A+K}),
we calculate it here using invariant theory methods for two reasons.
Firstly, we want to demonstrate the strength of the method introduced
in this paper. Secondly, in order to deal with the general case of
$G$-simple finite dimensional algebras we will be forced to rely
on some details of this specific proof. 

Let $H$ be a finite group and $A=F^{\mu}H=\oplus_{h\in H}Fb_{h}$
where $\mu\in H^{2}(H,F^{\times})$ (here $F^{\times}$ is considered
a trivial $H$-module) and $b_{h}$ is a formal element corresponding
to $h\in H$. Once again, we are interested in the group of all $H$-graded
automorphisms $\widetilde{H}$ and the multilinear invariants of $F[A^{\times n}]$
under its induced action (on $F[A^{\times n}]$). 
\begin{lem}
The following hold:
\begin{enumerate}
\item The linear character group of $H$ is isomorphic to $\widetilde{H}$.
More precisely, for a linear character $\chi$, the corresponding
automorphism is given by $f_{\chi}(b_{h})=\chi(h)b_{h}$.
\item The space $I_{n}$ of multilinear elements in $F[A^{\times n}]^{\widetilde{H}}$
is spanned by $\varepsilon_{1,h_{1}}\cdots\varepsilon_{n,h_{n}}$,
where $h_{1}\cdots h_{n}\in H'$ (the commutator subgroup of $H$)
and $\varepsilon_{i,g}(b_{h}^{(k)})=\delta_{g,h}\delta_{i,k}$ ($b_{h}^{(k)}=(0,...,0,b_{h},0,...,0)$,
where $b_{h}$ is placed in the $k$'th component).
\item $\dim_{F}I_{n}=|H'|\cdot|H|^{n-1}$.
\end{enumerate}
\end{lem}
\begin{proof}
It is clear that $f_{\chi}\in\widetilde{H}$. For the other direction,
suppose $f\in\widetilde{H}$. Hence, $f(b_{h})=\psi(h)b_{h}$ where
$\psi(h)\in F$. The multiplicativity of $f$ implies the multiplicativity
of $\psi$. Since $f$ is invertible, $\mbox{Im}\psi\in F^{\times}$.
Hence, $\psi$ is a linear character of $H$ and $f=f_{\psi}$- proving
(1).

For part (2), notice that $f_{\chi}^{-1}\cdot\varepsilon_{i,h}=\chi(h)\varepsilon_{i,h}$.
So, 
\[
f_{\chi}^{-1}\cdot\left(\varepsilon_{1,h_{1}}\cdots\varepsilon_{n,h_{n}}\right)=\chi(h_{1}\cdots h_{n})\varepsilon_{1,h_{1}}\cdots\varepsilon_{n,h_{n}}.
\]
We are done since $\chi(h_{1}\cdots h_{n})=1$ for every character
$\chi$ if and only if $h_{1}\cdots h_{n}\in H'$. 

Part (3) is clear from part (2). \end{proof}
\begin{rem}
As a result of the previous lemma, we \uline{identify} $\widetilde{H}$
with the group $\Hom(H,F^{\times})$ of linear characters of $H$. 
\end{rem}
The $F$-form $tr(b_{h})=\delta_{h,e}$ on $A$ is well known to satisfy
the conditions of \lemref{trace}. Therefore, considering the variables
$x_{i,h}$ of $R_{n}^{G}(A)=$ the relatively free $H$-graded $F$-algebra
of $A$ generated by $\{x_{i,h}|i=1...n,\, h\in H\}$, as $\varepsilon_{i,h}b_{h}\in F[A^{\times n}]\otimes A$,
we conclude that $\phi:C_{n}^{H}(A)\rightarrow I_{n+1}$ is a linear
isomorphism, where 
\[
\phi(x_{i_{1},h_{1}}\cdots x_{i_{n},h_{n}})=tr(x_{i_{1},h_{1}}\cdots x_{i_{n},h_{n}}x_{n+1,(h_{1}\cdots h_{n})^{-1}})=c\varepsilon_{i_{1},h_{1}}\cdots\varepsilon_{i_{n},h_{n}}\varepsilon_{n+1,(h_{1}\cdots h_{n})^{-1}}
\]
for a suitable non zero $c\in F$. As a result, our goal is to prove
that 
\[
\dim_{F}I_{n+1}\sim\dim_{F}\Tr(C_{n}^{H}).
\]

Let $RI_{n}=F[A^{\times n}]^{\widetilde{H}}$ and $RSI_{n}=\{tr(f)|f\in R_{n}(A)\}\subseteq RI_{n}$.
Note that these spaces have a natural structure as $GL_{n}(F)$-modules.
Moreover, the multilinear part of $RI_{n}$ and $RSI_{n}$ is exactly
$I_{n}$ and $\Tr(C_{n}^{H})$ respectively. Similarly to \secref{Hilbert-series}
we calculate the Hilbert series of $RI_{n}$ and prove that almost
all the coefficients of the Hilbert series of $RI_{n}$ and $RSI_{n}$
are equal, with respect to the Schur functions. 
\begin{lem}
The Hilbert series of $RI_{n}$ is 
\[
H_{RI_{n}}(t_{1},...,t_{n})=\sum_{\begin{array}{c}
\lambda\\
\htt(\lambda)\leq|H|
\end{array}}\left\langle s_{\lambda}\left(\chi(h)|h\in H\right),1\right\rangle s_{\lambda}(t_{1},...,t_{n}),
\]
where $\left\langle \cdot,\cdot\right\rangle $ is the inner product
in the ring of class function relative to the group $\widetilde{H}$
and $g_{\lambda}(\chi)=s_{\lambda}\left(\chi(h)|h\in H\right):\widetilde{H}\to F$
is a (class) function of $\widetilde{H}$.\end{lem}
\begin{proof}
This is done essentially as in (\cite{key-Formanek} pages 201,202).
Here are the details: First we calculate the $GL_{n}(F)\times\widetilde{H}$
character of the module $F[A^{\times n}]$, which we denote by $H_{F[A^{\times n}]}(t_{1},...,t_{n},\mbox{\ensuremath{\chi}})$.
It is known that characters of $GL_{n}(F)$ are determined by the
diagonal matrices, so since 
\[
diag(t_{1},...,t_{n})\times\mbox{\ensuremath{\chi\cdot}}(\varepsilon_{i_{1},h_{1}}\cdots\varepsilon_{i_{k},h_{k}})=\left(\chi(h_{1})t_{i_{1}}\cdots\chi(h_{k})t_{i_{k}}\right)\varepsilon_{i_{1},h_{1}}\cdots\varepsilon_{i_{k},h_{k}}
\]
it follows that 
\[
H_{F[A^{\times n}]}(t_{1},...,t_{n},\chi)=\prod_{i=1}^{n}\prod_{h\in H}\sum_{j=0}^{\infty}\left(\chi(h)t_{i}\right)^{j}
\]
which is equal to 
\[
\prod_{i=1}^{n}\prod_{h\in H}\frac{1}{1-\chi(h)t_{i}}=\sum_{\lambda}s_{\lambda}\left(\chi(h)|h\in H\right)s_{\lambda}(t_{1},...,t_{n}).
\]
Thus,
\[
H_{RI_{n}}(t_{1},...,t_{n})=\sum_{\lambda}\left\langle s_{\lambda}\left(\chi(h)|h\in H\right),1\right\rangle s_{\lambda}(t_{1},...,t_{n}).
\]
We are done since according to the definition of Schur functions $s_{\lambda}\left(\chi(h)|h\in H\right)=0$
for every $\lambda$ such that $\htt(\lambda)>|H|$. 
\end{proof}
Denote $a_{\lambda}=\left\langle s_{\lambda}\left(\chi(h)|h\in H\right),1\right\rangle $.
As before, we require an element $\mathcal{I}\in RI_{n_{0}}$ such
that:
\begin{equation}
\mathcal{I}RI_{n}\subseteq RSI_{n}\subseteq RI_{n}\label{eq:alpha}
\end{equation}
and $F\mathcal{I}$ is a one dimensional $GL_{n_{0}}(F)$-module.
Since $a_{\lambda}=0$ for $\htt(\lambda)>|H|$, we need to establish
\eqref{alpha} only for some $n$ which is $\geq|H|$. Furthermore,
the same reasoning as in \lemref{mu-la} shows: 
\[
a_{l\mu+\lambda}=a_{\lambda}.
\]
Thus, as in \secref{Hilbert-series} (see \corref{m=00003Dm'} and
\thmref{.glory}), we will be able to conclude the desired asymptotic
relation: 
\[
c_{n}^{H}(A)\sim\dim_{F}I_{n+1}=|H'||H|^{n}.
\]

So we are left with finding $\mathcal{I}$.
\begin{lem}
For $\mathbf{h}=(h_{1},...,h_{k})\in H^{k}$ denote by $\Sigma_{\mathbf{h}}$
the set of all elements in $H$ of the form $h_{\sigma(1)}\cdots h_{\sigma(k)}$,
where $\sigma\in S_{k}$. Then:
\begin{enumerate}
\item $|\Sigma_{\mathbf{h}}|\leq|H'|$.
\item If $\mathbf{h}_{1}$ is a sub-sequence of $\mathbf{h}_{2}$, then
$|\Sigma_{\mathbf{h}_{1}}|\leq|\Sigma_{\mathbf{h}_{2}}|$. In particular,
if $|\Sigma_{\mathbf{h}_{1}}|=|H'|$ then $|\Sigma_{\mathbf{h}_{2}}|=|H'|$.
\item There is some $n_{0}$ and $\mathbf{h}_{0}\in H^{n_{0}}$ for which
$\Sigma_{\mathbf{h}_{0}}=H'$.
\end{enumerate}
\end{lem}
\begin{proof}
The first part holds since the elements of $\Sigma_{\mathbf{h}}$
are all contained in the same $H'$-coset. The second part is trivial.
The third part follows easily from the fact that any element in $H'$
can be written as a product of commutators and from part 1. 
\end{proof}
Denote by $n$ the maximum between $|H|$ and $n_{0}$. 
\begin{lem}
Let $h_{1},...,h_{n_{0}}$ be the elements of $\mathbf{h}_{0}$ and
$h_{n_{0}+1}=\cdots=h_{n}=e$ . The element
\[
\mathcal{I}=\det\left(\begin{array}{ccc}
\varepsilon_{1,h_{1}} & \cdots & \varepsilon_{1,h_{n}}\\
\vdots & \ddots & \vdots\\
\varepsilon_{n,h_{1}} & \cdots & \varepsilon_{n,h_{n}}
\end{array}\right)
\]
satisfies \eqref{alpha}.\end{lem}
\begin{proof}
Notice that $\mathcal{I}$ is an alternating sum of the monomials
$Z_{\sigma}=\varepsilon_{\sigma(1),h_{1}}\cdots\varepsilon_{\sigma(n),h_{n}}$
(here $\sigma\in S_{n}$). Without loss of generality, it is sufficient
to show that for $Z=Z_{\sigma}$ and any monomial $T=\varepsilon_{i_{1},w_{1}}\cdots\varepsilon_{i_{k},w_{k}}\in RI_{n}$,
where $w_{1},...,w_{k}\in H$, $ZT\in RSI_{n}$. Indeed, since $w_{1}\cdots w_{k}\in H'$,
it follows from the previous Lemma (part 2) that it is possible to
permute $(w_{1},...,w_{k+n})=(w_{1},...,w_{k},h_{1},...,h_{n})$,
say by $\nu\in S_{k+n}$, so that $w_{\nu(1)}\cdots w_{\nu(k+n)}=e$.
Hence, 
\[
ZT=\varepsilon_{i_{\nu(1)},w_{\nu(1)}}\cdots\varepsilon_{t_{\nu(k+n)},w_{\nu(k+n)}}=tr\left(x_{i_{\nu(1)},w_{\nu(1)}}\cdots x_{t_{\nu(k+n)},w_{\nu(k+n)}}\right).
\]
Finally, it is trivial that $F\mathcal{I}$ is a $GL_{n_{0}}(F)$-module.
\end{proof}
Let us collect the main result of this section.
\begin{thm}
Let $H$ be a group, and $F$ be a field of characteristic zero. Suppose
$A=F^{\mu}H$, where $\mu\in H^{2}(H,F^{\times})$ (here $F^{\times}$
is considered a trivial $H$ module). Then 
\[
c_{n}^{H}(A)\sim|H'||H|^{n}.
\]

\end{thm}

\section{\label{sec:General-case}General case}

Let $A$ be a $G$-simple finite dimensional algebra. By \thmref{Let--be},
$A=F^{\mu}H\otimes M_{m}(F)$ where $A_{g}=\sp\left\{ b_{h}\otimes e_{\gamma_{i}(s),\gamma_{j}(t)}|\, g=\gamma_{i}^{-1}h\gamma_{j}\right\} $
and $\mathfrak{g}=(\gamma_{1}^{m_{1}},...,\gamma_{k}^{m_{k}})$ is
a grading vector of $M_{m}(F)$. In the previous sections the strategy
was to consider the multilinear invariants of $F[A^{\times n}]$ under
the group of $G$-graded automorphisms of $A$ as ``approximating
spaces'' for the codimension sequence $c_{n}^{G}(A)$. Then we (1)
calculate asymptotically the dimension sequence of the approximating
spaces, and (2) show that the Hilbert series of the non-multilinear
invariants invariants is almost the same as the Hilbert series corresponding
to the relatively free algebra of $A$. Applying this method in the
same fashion to the general case is rather complicated. Instead we
combine the previous results to obtain a crude evaluation of $c_{n}^{G}(A)$. 

\selectlanguage{american}%
Let $A_{1}=F^{\mu}H,\widetilde{H}=\Aut_{G}(A_{1})$, $A_{2}=M_{m}(F)$
graded by the vector $\mathfrak{g}=(\gamma_{1}^{m_{1}},...,\gamma_{k}^{m_{k}})$
and $\overline{G^{\prime}}=GL_{m_{1}}(F)\times\cdots\times GL_{m_{k}}(F)\subseteq GL_{m}(F)$
which acts on $A_{2}$ by conjugation. 

First, we \uline{claim} that $\widetilde{H}\times\overline{G^{\prime}}\subseteq\Aut_{G}(A)$,
where the action is given by $(\chi,X)\cdot b_{h}\otimes M=\chi(h)b_{h}\otimes X^{-1}MX.$
Indeed, $\chi$ acts as an automorphism on $A_{1}$ and $X$ on $A_{2}$,
so it is clear that $(\chi,X)$ acts as an automorphism on $A=A_{1}\otimes A_{2}$.
Moreover, the result of acting with $(\chi,X)$ on $b_{h}\otimes e_{t(i),s(j)}\in A_{t^{-1}hs}$
(here $h\in H$ and $t,s\in B=\{\gamma_{1},...,\gamma_{k}\}$) is
\[
\chi(h)b_{h}\otimes\widehat{X}_{t,t}^{-1}e_{t(i),s(j)}\widehat{X}_{s,s}\in A_{t^{-1}hs}.
\]
Because the elements $b_{h}\otimes e_{t(i),s(j)}$ form a $G$-graded
basis for $A$, we deduce that the above action also fixes the $G$-grading. 

As a result of this claim, (see \secref{Motivation}), $C_{n}^{G}(A)$
is embedded inside $\Hom(A^{\otimes n},A)^{\widetilde{H}\times\overline{G^{\prime}}}$.
Consider the trace on $A$ given by $tr(b_{h}\otimes M)=tr_{A_{1}}(b_{h})\cdot tr_{A_{2}}(M).$
Since $tr_{A_{1}}$ is non-degenerate on $A_{1}$ and $\widetilde{H}$-invariant,
and $tr_{A_{2}}$ is non degenerate on $A_{2}$ and $\overline{G^{\prime}}$-invariant,
we conclude that $tr$ is non-degenerate on $A$ and $\widetilde{H}\times\overline{G^{\prime}}$-invariant.
Hence, as in \lemref{trace}, $tr$ induces an isomorphism of $F$-spaces
between $\Hom(A^{\otimes n},A)^{\widetilde{H}\times\overline{G^{\prime}}}$
and 
\begin{eqnarray*}
T_{n}^{\widetilde{H}\times\overline{G^{\prime}}}(A)=\Hom\left(A^{\otimes n},F\right)^{\widetilde{H}\times\overline{G^{\prime}}} & = & \Hom\left(A_{1}^{\otimes n}\otimes A_{2}^{\otimes n},F\right)^{\widetilde{H}\times\overline{G^{\prime}}}\\
 & = & \Hom(A_{1}^{\otimes n},F)^{\widetilde{H}}\otimes\Hom(A_{2}^{\otimes n},F)^{\overline{G^{\prime}}}.
\end{eqnarray*}
Therefore, by the results in the previous sections, we obtain the
\foreignlanguage{english}{upper bound:
\[
c_{n}^{G}(A)\leq\dim_{F}T_{n+1}^{\widetilde{H}\times\overline{G^{\prime}}}(A)\leq\dim_{F}T_{n+1}^{\widetilde{H}}(A_{1})\cdot\dim_{F}T_{n+1}^{\overline{G^{\prime}}}(A_{2})\sim\delta n^{\frac{1-\dim_{F}A_{e}}{2}}\left(|H|m^{2}\right)^{n}
\]
($\delta$ is some constant). }

\selectlanguage{english}%
Next, consider the algebra: 
\[
\mathcal{A}=A\otimes_{F}F\left[\varepsilon_{r,h},u_{t(i),s(j)}^{(r)}\,|\, r=1,2,...;\, h\in H;\, s,t\in B;\,1\leq i\leq m_{t};\,1\leq j\leq m_{s}\right]
\]
and the generic elements: 
\[
U_{r,g}=\sum_{t,s\in B;\, g=t^{-1}hs}\sum_{i,j}\varepsilon_{r,h}u_{t(i),s(j)}^{(r)}\cdot b_{h}\otimes e_{t(i),s(j)}\in\mathcal{A}.
\]

\begin{lem}
The map sending $x_{r,g}$ to $U_{r,g}$ induces an isomorphism of
$F$-vector spaces between $C_{n}^{G}(A)$ and $\mathbf{C}_{n}^{G}(A)=\sp\left\{ U_{\sigma(1),g_{\sigma(1)}}\cdots U_{\sigma(n),g_{\sigma(n)}}\thinspace|\,\sigma\in S_{n};\, g_{1},...,g_{n}\in G\right\} $.\end{lem}
\begin{proof}
Since $\mathcal{A}$ is a scalar extension of $A$ by a commutative
$F$-algebra, the map induces a map from $C_{n}^{G}(A)$ to $\mathbf{C}_{n}^{G}(A)$.
Moreover, if $f(U_{1,g_{1}},...,U_{n.g_{n}})=0$, where $f=f(x_{1,g_{1}},...,x_{n,g_{n}})$
is a $G$-graded multilinear polynomial. Then,
\[
0=\sum\varepsilon_{1,h_{1}}\cdots\varepsilon_{n,h_{n}}u_{t_{1}(i_{1}),s_{1}(j_{1})}^{(1)}\cdots u_{t_{n}(i_{n}),s_{n}(j_{n})}^{(n)}f(b_{h_{1}}\otimes e_{t_{1}(i_{1}),s_{1}(j_{1})},...,b_{h_{n}}\otimes e_{t_{n}(i_{n}),s_{n}(j_{n})}),
\]
where the sum is over all suitable $h_{1},...,h_{n};s_{1},...,s_{n};t_{1},...,t_{n};i_{1},...,i_{n}$
and $j_{1},...,j_{n}$. Since the coefficients of different $f(b_{h_{1}}\otimes e_{t_{1}(i_{1}),s_{1}(j_{1})},...,b_{h_{n}}\otimes e_{t_{n}(i_{n}),s_{n}(j_{n})})$
are also different, we get that all $f(b_{h_{1}}\otimes e_{t_{1}(i_{1}),s_{1}(j_{1})},...,b_{h_{n}}\otimes e_{t_{n}(i_{n}),s_{n}(j_{n})})$
involved are zero. This shows, since $f$ is multilinear, that $f$
is an identity of $A$. Hence, the map in the lemma is injective.
Because it is clearly onto , we are done.
\end{proof}
To proceed we assume the following two assumption:
\begin{enumerate}
\item $\gamma_{1}=e$.
\item $H\gamma_{1},...,H\gamma_{k}$ are distinct $H$-cosets.
\end{enumerate}
(see in \cite{ELI+Darrel} that there is no loss of generality.) Therefore,
for a given $g\in G$ and $t\in B$ there could be no more than one
pair $(h,s)\in H\times B$ for which $g=t^{-1}hs$. Indeed, since
$hs=tg$, $hs$ must be the unique way to write the element $tg$
as a left $H$-coset. 
\begin{rem}
\label{rem:simplier_multi}As a result, if $I_{t}=\sum_{i}e_{t(i),t(i)}$
(here $t\in B$), then $I_{t}\cdot U_{r,g}=\sum_{i,j}\varepsilon_{r,h}u_{t(i),s(j)}^{(r)}\cdot b_{h}\otimes e_{t(i),s(j)}$,
where $g=t^{-1}hs$ (if there are no such $s\in B$ and $h\in H$,
the expression is zero). 

Furthermore, 
\[
I_{t}\cdot U_{r_{1},g_{1}}\cdots U_{r_{n},g_{n}}=\sum\varepsilon_{r_{1},h_{1}}\cdots\varepsilon_{r_{n},h_{n}}u_{t_{1}(i_{1}),t_{2}(i_{2})}^{(r_{1})}\cdots u_{t_{n}(i_{n}),t_{n+1}(i_{n+1})}^{(r_{n})}\cdot b_{h_{1}\cdots h_{n}}\otimes e_{t(i_{1}),t_{n+1}(j_{n+1})},
\]
where the sum is taken over all suitable $i_{1},...,i_{n+1}$ and
$t_{1},...,t_{n+1}\in B;\, h_{1},...,h_{n}\in H$ satisfy:
\[
t_{1}=t;g_{1}=t_{1}^{-1}ht_{2};...;g_{n}=t_{n}^{-1}h_{n}t_{n+1}.
\]

\end{rem}
\selectlanguage{american}%
It is clear that $\sp\left\{ U_{n+1,e}U_{\sigma(1),g_{1}}\cdots U_{\sigma(n),g_{n}}\,|\,\sigma\in S_{n},\,\,(g_{1},...,g_{n})\in G^{n}\right\} $
is contained in $C_{n+1}^{G}(A)$. Moreover, we may project the above
spaces \textbf{onto }$C_{e,n}^{G}(A)=\sp\left\{ I_{e}U_{\sigma(1),g_{1}}\cdots U_{\sigma(n),g_{n}}\,|\,\sigma\in S_{n},\,\,(g_{1},...,g_{n})\in G^{n}\right\} $,
where $I_{e}=\sum_{i}e_{e(i),e(i)}$ (recall that $e=\gamma_{1}\in B$).
Therefore, $\dim_{F}C_{e,n}^{G}(A)\leq\dim_{F}C_{n+1}^{G}(A)$. By
\lemref{trace}, 
\[
\psi(f)=tr(f\cdot\sum_{g\in G}U_{n+1,g})\in\Hom(A^{\otimes(n+1)},A),
\]
where $f\in C_{e,n}^{G}(A)$, is an embedding. Denote the image of
this map by $QI_{n}$. 

Now, choose $(h_{2},...,h_{n})\in H^{n-1}$, $\{g_{2},...,g_{n}\}\in B^{n-1}$
and a cyclic permutation $\sigma=(\sigma_{n+1}=\sigma_{2}\,...\,\sigma_{n})\in S_{\{2,...,n\}}$.
The element 
\[
\varepsilon_{1,e}\varepsilon_{2,h_{2}}\cdots\varepsilon_{n,h_{n}}\varepsilon_{n+1,h_{n+1}}\cdot\sum_{i_{0},...,i_{n}}u_{e(i_{0}),g_{\sigma_{2}}(i_{\sigma_{2}})}^{(1)}u_{g_{2}(i_{2}),g_{\sigma(2)}(i_{\sigma(2)})}^{(2)}\cdots u_{g_{n}(i_{n}),g_{\sigma(n)}(i_{\sigma(n)})}^{(n)}u_{g_{\sigma_{n}}(i_{\sigma_{n}}),e(i_{0})}^{(n+1)},
\]
where $h_{n+1}=(h_{\sigma_{2}}\cdots h_{\sigma_{n}})^{-1}$ , is inside
$QI_{n}$. Indeed, using \remref{simplier_multi}, it is not hard
to verify that the above expression is the image by $\psi$ of 
\[
I_{e}\cdot U_{1,g_{\sigma_{2}}}\cdot U_{\sigma_{2},g_{\sigma_{2}}^{-1}h_{\sigma_{2}}g_{\sigma_{3}}}\cdots U_{\sigma_{n},g_{\sigma_{n}}^{-1}h_{\sigma_{n}}g_{\sigma_{n+1}}}.
\]
Finally, by evaluating $\varepsilon_{1,e}=\varepsilon_{n+1,h}=u_{e(i),t(j)}^{(1)}=u_{s(j^{\prime}),e(i)}^{(n+1)}=1$
for all $h\in H,s,t\in B$ and suitable $i,j,j^{\prime}$, we obtain
a linear space with even smaller dimension containing all the elements
of the form 
\[
\varepsilon_{2,h_{2}}\cdots\varepsilon_{n,h_{n}}\cdot\sum_{i_{2},...,i_{n}}u_{g_{2}(i_{2}),g_{\sigma(2)}(i_{\sigma(2)})}^{(2)}\cdots u_{g_{n}(i_{n}),g_{\sigma(n)}(i_{\sigma(n)})}^{(n)}.
\]
The dimension is at least $|H|^{n-1}\cdot\dim_{F}SI_{n-1}(A_{2})\sim\beta n^{\frac{1-\dim_{F}A_{e}}{2}}\left(|H|m^{2}\right)^{n}$
for some constant $\beta$ (see \secref{Hilbert-series}).

\selectlanguage{english}%
Thus, by a result of the second author claiming that $c_{n}^{G}(A)\sim\alpha n^{-b}d^{n}$
for some constants $\alpha,b,d$, (see \cite{Shpigelman}), we conclude:
\begin{thm*}[\textbf{B}]
For every finite dimensional $G$-simple algebra $A$, there is a
constant $\alpha$ such that
\end{thm*}
\[
c_{n}^{G}(A)\sim\alpha n^{\frac{1-\dim_{F}A_{e}}{2}}\left(\dim_{F}A\right)^{n}.
\]

\section{Example}

In this section we demonstrate that the asymptotics of the codimension
sequence is incapable of distinguishing between two non-isomorphic
$G$-gradings. In other words, we address the question whether different
asymptotics of the codimension sequence implies different $G$-gradings.
Without any assumptions on the $G$-gradings, it is rather easy to
construct trivial counterexamples even with the same codimension sequence.
For instance, take any $H$-graded algebra $A$ and consider the following
$G=H\times H$-gradings: $A_{(h,t)}^{(1)}=\begin{cases}
A_{h} & ,t=e\\
0 & \mbox{,otherwise}
\end{cases}$ and $A_{(t,h)}^{(2)}=\begin{cases}
A_{h} & ,t=e\\
0 & \mbox{,otherwise}
\end{cases}$. 

Therefore, we seek after two $G$-graded algebras $A^{(1)}$ and $A^{(2)}$
having the same $G$-graded codimension asymptotics, but 
\begin{itemize}
\item $\mbox{supp}_{G}A^{(1)}=\mbox{supp}_{G}A^{(2)}=G$ (recall that $\mbox{supp}_{G}A$
is the subgroup of $G$ generated by all $g\in G$ for which $A_{g}\neq0$).
\item There is no $F$-algebra isomorphism $\phi:A^{(1)}\rightarrow A^{(2)}$
and group automorphism $\psi:G\to G$ such that $\phi(A_{g}^{(1)})=A_{\psi(g)}^{(2)}$
for every $g\in G$ (note that in particular we demand here that the
algebras are not $G$-isomorphic).
\end{itemize}
Using \thmref{A} we are able to give $A^{(1)}$ and $A^{(2)}$ satisfying
these restrictions. Let $G=D_{3}=\left\langle s,r|s^{2}=e,r^{3}=e,srs=r^{-1}\right\rangle $
the dihedral group of order 6 and $A=M_{6}(F)$. Consider $A^{\mathfrak{g}}$
and $A^{\mathfrak{h}}$-two elementary gradings of $A$ corresponding
to the vectors $\mathfrak{g}=\left(e,e,e,s,s,r\right)$ and $\mathfrak{h}=\left(e,e,e,r,r,s\right)$.
That is,
\[
A^{\mathfrak{g}}=\begin{pmatrix}e & e & e & s & s & r\\
e & e & e & s & s & r\\
e & e & e & s & s & r\\
s & s & s & e & e & sr\\
s & s & s & e & e & sr\\
r^{2} & r^{2} & r^{2} & sr & sr & e
\end{pmatrix};\,\,\, A^{\mathfrak{h}}=\begin{pmatrix}e & e & e & r & r & s\\
e & e & e & r & r & s\\
e & e & e & r & r & s\\
r^{2} & r^{2} & r^{2} & e & e & sr\\
r^{2} & r^{2} & r^{2} & e & e & sr\\
s & s & s & sr & sr & e
\end{pmatrix}
\]

This example satisfies the two restriction mentioned above: By definition
$A_{s}^{\mathfrak{g}},A_{r}^{\mathfrak{g}}\ne\{0\}$ and $A_{s}^{\mathfrak{h}},A_{r}^{\mathfrak{h}}\ne\{0\}$,
thus the support of both gradings is $G$. Moreover, it is easy to
see that $\dim_{F}A_{s}^{\mathfrak{g}}=12$, but there is no component
in the other grading of dimension $12$, thus there are no $\phi$
and $\psi$ such that $\phi\left(A_{s}^{\mathfrak{g}}\right)=A_{\psi(s)}^{\mathfrak{h}}$.
Finally, note that in both gradings $m_{1}=3$, $m_{2}=2$ and $m_{3}=1$,
thus $H_{\mathfrak{g}}=H_{\mathfrak{h}}=\{e\}$. Hence by \thmref{A}
the asymptotics of the two $G$-graded codimension sequences is exactly
\[
\alpha n^{-\frac{13}{2}}6^{2n},
\]
where
\[
\alpha=\frac{6^{9}}{2^{6}\sqrt{3\left(2\pi\right)^{5}}}.
\]


\begin{thebibliography}{AGL}
\bibitem[AH]{ELI+Darrel}E. Aljadeff and D. Haile, Simple G-graded
algebras and their polynomial identities, Trans. Amer. Math. Soc.
366 (2014), no. 4, 1749\textendash 1771.

\bibitem[AG]{key-A+G}E. Aljadeff and A. Giambruno, Multialternating
graded polynomials and growth of polynomial identities, to appear
in Proc. Amer. Math. Soc.

\bibitem[AGL]{key-A+G+L}E. Aljadeff; A. Giambruno and D. La Mattina,
Graded polynomial identities and exponential growth, J. Reine Angew.
Math 650 (2011), 83\textendash 100.

\bibitem[AK1]{key-1}E. Aljadeff and A. Kanel-Belov, Representability
and Specht problem for G-graded algebras, Adv. in Math. 225 (2010),
pp. 2391\textendash 2428.

\bibitem[AK2]{key-A+K}E. Aljadeff and A. Kanel-Belov, Hilbert series
of PI relative free G-graded algebras are rational functions, to appear
in Bulletin of London Math Society.

\bibitem[BSZ]{key-5}Yu. A. Bahturin, S. K. Sehgal and M. V. Zaicev,
Finite-dimensional simple graded algebras, Sb. Math. 199 (2008), no.
7, 965\textendash 983.

\selectlanguage{american}%
\bibitem[BR]{Beckner=000026Regev}Beckner, William; Regev, Amitai
Asymptotic estimates using probability. Adv. Math. 138 (1998), no.
1, 1\textendash 14.

\bibitem[Be1]{key-18E}\foreignlanguage{english}{A. Berele, Properties
of hook Schur functions with applications to PI. algebras, (English
summary) Adv. in Appl. Math. 41 (2008), no. 1, 52\textendash 75.}

\selectlanguage{english}%
\bibitem[Be2]{key-Berele}A. Berele. Homogeneous polynomial identities,
Israel J. Math. 42 (1982). 258-272

\bibitem[BeR]{key-19E}A. Berele and A. Regev, Asymptotic behaviour
of codimensions of PI. algebras satisfying Capelli identities, Trans.
Amer. Math. Soc. 360 (2008), no. 10, 5155\textendash 5172.

\selectlanguage{american}%
\bibitem[Bru]{key-Bruce}\foreignlanguage{english}{ Bruce .E. Sagan,
The Symmetric Group: Representations, Combinatorial Algorithms and
Symmetric Functions, Section 4.4, 155-157}

\selectlanguage{english}%
\bibitem[D]{key-Drensky}V. S. Drensky. Codimensions of T-ideals and
Hilbert series of relatively free algebras. C. R. Acad. Bulgare Sci.
34 (1981), 1201-1204.

\selectlanguage{american}%
\bibitem[F1]{key-Formanek}\foreignlanguage{english}{ E. Formanek,
Invariants and the ring of generic matrices, J. Algebra 89 (1984),
178-223.}

\selectlanguage{english}%
\bibitem[F2]{key-F2}E. Formanek, A conjecture of Regev about the
Capelli polynomial, J. Algebra 109 (1987), 93-114.

\selectlanguage{american}%
\bibitem[GZ1]{key-10}\foreignlanguage{english}{A. Giambruno and M.
Zaicev, Polynomial identities and asymptotic methods. Mathematical
Surveys and Monographs, 122. American Mathematical Society, Providence,
RI, 2005.}

\selectlanguage{english}%
\bibitem[GL]{key-G+L}A. Giambruno and D. La Mattina, Graded polynomial
identities and codimensions: computing the exponential growth. Adv.
Math. 225 (2010), no. 2, 859\textendash 881. 

\bibitem[Gr]{key-Green}J. A. Green. \textquotedblleft Polynomial
Representations of CL,.\textquotedblright{} Lecture Notes in Mathematics
No. 830, Springer-Verlag, Berlin/Heidelberg/New York, 1980.

\bibitem[HN]{key-11}D. Haile and M. Natapov, A graph theoretic approach
to graded identities, J. Algebra 365 (2012)

\bibitem[KRV]{key-13}A. Kanel-Belov, L.H. Rowen and U. Vishne, Structure
of Zariski-closed algebras, Trans. Amer. Math. Soc. 362 (2010), no.
9, 4695\textendash 4734.

\bibitem[Ke]{key-14}A. R. Kemer, Ideals of Identities of Associative
Algebras, Amer. Math. Soc., Translations of Monographs 87 (1991).

\bibitem[KL]{key-15}G. R. Krause and T. H. Lenagan, Growth of algebras
and Gel\textquoteright fand-Kirillov dimension. Research Notes in
Mathematics, 116. Pitman (Advanced Publishing Program), Boston, MA,
1985.

\bibitem[M]{key-MacBook}I.G. MacDonald, Symmetric Functions and Hall
Polynomials, The Clarendon Press, Oxford University Press, New York,
1979 

\bibitem[P]{key-Procesi}C. Procesi, The invariant theory of $n\times n$
matrices, Adv. Math. 19 (1976), 306-381.

\bibitem[Ra]{key-Raz}Ju. P. Razmyslov, Trace identities of full matrix
algebras over a field of characteristic zero, Math USSR Izv. 8 (1974),
727-760.

\bibitem[Re1]{key-Regev2}A. Regev, Asymptotic values for degrees
associated with strips Young diagrams, Adv. Math. 41 (2) (1981), 115-136.

\bibitem[Re2]{key-Regev}A. Regev, Codimensions and trace codimensions
of matrices are asymptotically equal, Israel J. Math 47 (1984), 246-250.

\bibitem[Re3]{key-46}A. Regev, Existence of identities in $A\otimes B$,
Israel J. Math. 11 (1972), 131\textendash 152.

\bibitem[Re4]{key-1}A. Regev, The representations of wreath products
via double centralizing theorems. J. Algebra 102 (1986), no. 2, 423\textendash 443. 

\bibitem[RS]{key-Rich}L. B. Richmond and J. Shallit, Counting abelian
squares, Electronical J. of Combinatorics 16 (2009), \#R72.

\bibitem[S]{Shpigelman}Y. Shpigelman The asymptotic behavior of the
codimension sequence of affine $G$-graded algebras. preprint 

\bibitem[SSW]{key-20}L. W. Small, J. T. Stafford and R. B. Warfield
Jr., Affine algebras of Gel\textquoteright fand-Kirillov dimension
one are PI. Math. Proc. Cambridge Philos. Soc. 97 (1985), no. 3, 407\textendash 414.

\bibitem[SW]{key-19}L. W. Small and R. B. Warfield Jr., Prime affine
algebras of Gel\textquoteright fand-Kirillov dimension one. J. of
Algebra 91 (1984), no. 2, 386\textendash 389.\end{thebibliography}
\end{document}